\documentclass[reqno,11pt]{amsart}

\usepackage[T1]{fontenc} 
\usepackage[utf8]{inputenc} 
\usepackage{mathtools}
\usepackage{amssymb}
\usepackage{amsthm}
\usepackage{bbm}
\usepackage{mathrsfs}
\usepackage[english]{babel} 
\usepackage{enumitem}
\usepackage{dsfont}
\usepackage{url}

\usepackage{hyperref}
\hypersetup{linktocpage=true, colorlinks=true, breaklinks=true, linkcolor=blue, menucolor=black, urlcolor=black,citecolor=blue,filecolor=green}

\newcommand{\N}{\mathbb{N}} 
\newcommand{\Q}{\mathbb{Q}} 
\newcommand{\R}{\mathbb{R}} 

\newcommand{\1}{\mathds{1}}

\theoremstyle{plain}
\newtheorem{thrm}{Theorem}[section]
\newtheorem{prop}[thrm]{Proposition}
\newtheorem{cor}[thrm]{Corollary}
\newtheorem{lemma}[thrm]{Lemma}

\theoremstyle{definition}

\newtheorem{remark}[thrm]{Remark}

\def\N{\bbN}

\numberwithin{equation}{section}

\DeclareMathSymbol{\leqslant}{\mathalpha}{AMSa}{"36}
\DeclareMathSymbol{\geqslant}{\mathalpha}{AMSa}{"3E}
\DeclareMathSymbol{\doteqdot}{\mathalpha}{AMSa}{"2B}
\DeclareMathSymbol{\circlearrowright}{\mathalpha}{AMSa}{"08}
\DeclareMathSymbol{\subsetneq}{\mathalpha}{AMSb}{"28}
\DeclareMathSymbol{\supsetneq}{\mathalpha}{AMSb}{"29}
\renewcommand{\leq}{\;\leqslant\;}
\renewcommand{\geq}{\;\geqslant\;}

\newcommand{\dd}{{\rm d}}
\newcommand{\e}[1]{\,{\rm e}^{#1}\,}
\newcommand{\ii}{{\rm i}}

\DeclareMathOperator*{\supp}{\text{supp}}

\newcommand{\upchi}{\raise 2pt \hbox{$\chi$}}

\def\writefig#1 #2 #3 {\rlap{\kern #1 truecm \raise #2 truecm
		\hbox{#3}}}

\newcommand{\caA}{{\mathcal A}}

\newcommand{\caF}{{\mathcal F}}
\newcommand{\caG}{{\mathcal G}}

\newcommand{\caW}{{\mathcal W}}

\def\bbone{{\mathchoice {\rm 1\mskip-4mu l} {\rm 1\mskip-4mu l} {\rm 1\mskip-4.5mu l} {\rm 1\mskip-5mu l}}}

\newcommand{\bbE}{{\mathbb E}}

\newcommand{\bbN}{{\mathbb N}}

\newcommand{\bbP}{{\mathbb P}}

\newcommand{\bbR}{{\mathbb R}}

\begin{document}
	
	\title{A functional central limit theorem for Polaron path measures}
	
	\author{Volker Betz}
	\address{Volker Betz \hfill\newline
		\indent FB Mathematik, TU Darmstadt}
	\email{betz@mathematik.tu-darmstadt.de}
	
	\author{Steffen Polzer}
	\address{Steffen Polzer \hfill\newline
		\indent FB Mathematik, TU Darmstadt
	}
	\email{polzer@mathematik.tu-darmstadt.de}

	\maketitle

	\begin{quote}
		{\small
			{\bf Abstract.}
			The application of the Feynman-Kac formula to Polaron models of quantum theory leads to 
			the path measure of Brownian motion perturbed by a pair potential that is translation 
			invariant both in space and time. An important problem in this context is the validity of a central 
			limit theorem in infinite volume. We show both the existence of the relevant infinite volume limits and 
			a functional central limit theorem in a generality that includes the Fröhlich 
			polaron for all coupling constants. The proofs are based on an extension of a novel method by Mukherjee 
			and Varadhan. 
		}  
		
		\vspace{1mm}
		
		\vspace{1mm}
		\noindent
		{\footnotesize {\it 2020 Math.\ Subj.\ Class.:} 60F17, 60G55, 60K05, 81S40}
	\end{quote}
	
	\section{Introduction} 
	
	The polaron is a model for a quantum particle interacting with a polar crystal. The interaction affects 
	at least two quantities of physical interest. On the one hand, the ground state energy of the system 
	is lowered when the interaction is increased. On the other hand, the particle needs to drag along the polarization 
	when it moves, and thus appears heavier than it would be without the interaction: it has an effective mass larger 
	than its bare mass.  For a particular polaron model now known as the Fröhlich polaron, 
	both of these effects were investigated by Feynman in \cite{Fe55} using his newly invented path integral method.  
	The central object in this context is the quantity 
	\begin{equation}
	\label{eq:Feynmans}  
	z_t = \int \exp \Big( \alpha \int_0^t \int_0^t  \frac{\e{-|s-r|}}{|\mathbf x_s - \mathbf x_r|} \, \dd s \, \dd r  \Big)  
	\caW_0 (\dd \mathbf x),
	\end{equation} 
	where $\caW_0$ is three-dimensional Brownian motion, and $\alpha \geq 0$ is the coupling constant determining the 
	strength of the interaction between the particle and the polar crystal. For the ground state energy, 
	Feynman argues that there exists $E_\alpha \leq 0$ such that $z_t \sim \e{-E_\alpha t}$ as $t \to \infty$, 
	and that $E_\alpha$ is then ground state energy of the polaron. For the effective mass, he finds that it 
	can be obtained by pinning the Brownian motion at $\bar x$ at the final time $t$ and observing the dependence of 
	$z_t$ on $\bar x$ under these circumstances. More precisely, when $z_t(\bar x)$ is the value of the integral
	\eqref{eq:Feynmans} with an additional factor of $\delta(\mathbf x_t - \bar x)$ in the integrand, 
	then $z_t$ should behave like $\exp(- E_\alpha t - m_\alpha \bar x^2/2t)$ for large $t$, where $m_\alpha$ is the 
	effective mass. Feynman then goes on to derive quantitative estimates for the ground state energy and the effective 
	mass as functions of $\alpha$. Among other things, he finds that  
	$E_\alpha \leq - \alpha - 1.23 (\alpha/10)^2$, and that $m_\alpha \sim 202 (\alpha/10)^4$ for large $\alpha$, and 
	remarks that these results compare well with those obtained earlier by Pekar \cite{Pe49} using the adiabatic 
	approximation. 	
	
	The arguments of Feynman, while ingenious, are almost completely non-rigorous. Since then, there have been many 
	efforts to understand various aspects of the polaron problem in a mathematically rigorous way. Most fundamental 
	among them is a firm connection between formula \eqref{eq:Feynmans} and the underlying many-body quantum system. 
	This connection has emerged over the years in various forms, and it is difficult to track it back to a single source. 
	We refer to \cite{Mo06} for a review on functional analytic aspects of the polaron, and to  \cite{DySp20} 
	for an outline on the connection to probability theory. For the convenience of the reader, we also 
	and include a short overview on the topic as an appendix to the present paper. 
	
	The problem of the ground state energy was completely solved by Donsker and Varadhan in \cite{DoVa83} at least for 
	the (physically most relevant) limit of large $\alpha$. Using their large deviation techniques, they find that both  
	$\psi(\alpha) := \lim_{t \to \infty} \frac{1}{t} \log z_t$ and 
	$g_0 := \lim_{\alpha \to \infty} \psi(\alpha)/\alpha^2$ exist, and that $g_0$ can be calculated using the variational
	formula of Pekar \cite{Pe49}. Lieb and Thomas \cite{LiTh97} give a functional analytic proof of 
	the same result that in addition yields explicit error estimates.  
	
	The problem of the effective mass turned out to be more difficult. Spohn \cite{Sp87} observed that the mathematically 
	rigorous connection between the effective mass and the quantity $z_t$ should appear through a central limit theorem: 
	for each $t > 0$, 
	one can read \eqref{eq:Feynmans} as the normalization of a probability measure $\bbP_t$ on path space, 
	namely the perturbation of Brownian motion by the exponentiated double integral; see formula \eqref{pair potential}
	below. 
	The task is then to prove the existence of a  limiting probability measure $\bbP_\infty$ as $t \to \infty$, 
	and a central limit theorem under diffusive rescaling of $\bbP_\infty$. The emerging diffusion constant 
	is the inverse of the effective mass. 
	There are some details that need consideration, such as the sense in which the limit exists, given that the 
	interaction is translation invariant. While the treatment in \cite{Sp87} is still partly non-rigorous, 
	recently Dybalski and Spohn \cite{DySp20} put the connection between a central limit theorem 
	and the effective mass on solid mathematical ground. 
	
	This leaves the problem of actually showing the central limit theorem, and as a prerequisite 
	the existence of an infinite volume measure. 
	A main obstacle is the singularity in the integrand of the double integral that prevents e.g.\ 
	uniform estimates in the paths. In this paper, we show for the first time that for all values of the coupling
	constant $\alpha$, infinite volume measures related to \eqref{eq:Feynmans} exist, and that 
	a functional central limit theorem holds. We also give an explicit expression for the diffusion constant that 
	is explicit enough to read off its strict positivity immediately. 
	
	Our results actually hold for more general models than the Fröhlich polaron. As discussed in the Appendix, 
	various polaron models of quantum theory give rise to probability measures of the type 
	\begin{equation}
	\label{pair potential}
	\mathbb P_{\alpha, T}(\mathrm d\mathbf x) = \frac{1}{Z_{\alpha, 2T}}\exp\bigg( \frac{\alpha}{2} 
	\int_{-T}^T\int_{-T}^T w(|t-s|, \mathbf x_t -\mathbf x_s) \, \mathrm ds \mathrm dt \bigg) \mathcal 
	W(\mathrm d \mathbf x),
	\end{equation}
	where $\alpha, T>0$, and where $w$ depends on the specific model, $Z_{\alpha,2T}$ is the normalizing constant, 
	or partition function, and $\caW$ is the path measure of Brownian motion, 
	although here we have to be a bit careful: the translation invariance of the interaction means that when we e.g.\ let $\caW$ be the 
	distribution of Brownian motion started at time $-T$ in the origin, then the distribution of $\mathbf x \mapsto \mathbf x_0$ under $\mathbb P_{\alpha, T}$ can not be expected to 
	converge. An elegant solution is to let $\caW$ be the measure of two-sided Brownian motion pinned 
	to $0$ at time $0$, and to consider the family $(\mathbb P_{\alpha, T})_{T \geq 0}$ on the $\sigma$-algebra 
	$\mathcal A$ generated by the increments $X_{s,t}(\mathbf x) = \mathbf x_t - \mathbf x_s$. 
	This choice gets rid of the rather arbitrary 
	pinning at $t=0$, and the measures on the restricted $\sigma$-algebra will be shown to converge locally 
	in the usual sense. 
    
	We call  measures of the form \eqref{pair potential} {\em Polaron path measures}. The precise conditions
	(A1)-(A3) that we impose on $w$ will be stated at the beginning of Section \ref{sec:results}. The most important 
	among them is the growth condition 
	\begin{equation}
	\label{Our condition}
	\limsup_{T \to \infty} \frac{Z_{\alpha, 2T}}{Z_{\alpha, T}^2} <\infty \tag{GC}
	\end{equation}
	on the partition function. For suitable $w$, we find a number of equivalent conditions to \eqref{Our condition}, one of them 
	being the existence and strict positivity of $\lim_{T \to \infty} Z_{\alpha,T} \e{-\psi(\alpha)T}$, where
	\[
	\psi(\alpha) = \lim_{T \to \infty} \frac{1}{T} \log Z_{\alpha,T}
	\]
	is the free energy associated to \eqref{pair potential}.  Under these conditions, we show the existence of an 
	infinite volume limit $\bbP_{\alpha,\infty}$ of the family $(\mathbb P_{\alpha, T})_{T \geq 0}$. 
	Under the additional condition that $x \mapsto w(t,x)$ is quasiconcave for all $t \geq 0$ we show a 
	functional central limit theorem for $\bbP_{\alpha,\infty}$ under diffusive rescaling. 		
		
	Central limit theorems for similar models have been proved before. \cite{BeSp04} focuses on
	Polaron path measures that originate from quantum polaron models. 
	The method rests on not integrating out the field, and instead working directly with a Markov 
	process on an infinite dimensional state space; the central limit theorem is then obtained by studying the 
	quantum field as seen from the moving particle, and by applying a suitable theory of Kipnis and Varadhan 
	\cite{KiVa86}. The method needs, among other assumptions, that $w$ is bounded, and 
	thus does not work for the Fröhlich polaron. 
	
	Gubinelli \cite{Gu06} treats general Polaron path measures, and implements an approach that had already 
	been proposed in \cite{Sp87}: By cutting $[-T,T]$
	into blocks of finite length and considering the space of continuous functions on each block, 
	he obtains an one-dimensional system of infinite dimensional spins, where  all spins interact 
	with each other through the double integral, but in case of sufficiently fast decay of $t \mapsto w(t, x)$, 
	that interaction becomes weak for distant blocks. Dobrushins theory of one-dimensional spin
	systems \cite{Do68, Do70} then guarantees existence and uniqueness of the infinite volume limit, 
	and the central limit theorem is a consequence of sufficiently fast decay of correlations between distant blocks. 
	The conditions on $w$ are weaker than those from \cite{BeSp04} in some aspects, but stronger in others. They also need 
	uniformly bounded potentials and thus exclude the Fröhlich polaron. Also, the coupling 
	constant $\alpha$ needs to be sufficiently small for Dobrushins method to work. 
	
	The method of Mukherjee \cite{Mu17} also relies on cutting the space of functions into blocks. 
	He uses it to construct a discrete time Markov chain (with infinite dimensional state space) 
	that describes the behaviour of the path in a given block given its behaviour in the previous block, and shows 
	that this Markov chain fulfills the conditions for a generalized Perron-Frobenius argument. 
	The central limit theorem is then derived as a consequence of the CLT for additive functionals of 
	stationary Markov chains. The method works most cleanly when $t \mapsto w(t, x)$ has compact support, 
	in which case it can also handle some specific unbounded $w$. It also works when 
	$\sup_x |w(t, x)| \leq C t^{-2-\delta}$ for constants $C,\delta>0$. Both conditions exclude the Fröhlich polaron.

	Important progress was achieved recently by Mukherjee and Varadhan in the work \cite{MV19} mentioned already above. 
	The key idea is not  
	technically difficult, but ingenious: for the case of the Fröhlich polaron, one expands the exponential in 
	\eqref{pair potential} into a Taylor series, exchanges the 
	order of integration, and interprets \eqref{pair potential} as a mixture of Gaussian measures, where the 
	mixing measure is a point process on the space of (possibly overlapping) 
	finite subintervals of $\bbR$. 
	The results of the paper are the existence of $\bbP_{\alpha,\infty}$ and a central 
	limit theorem, valid for either large enough or small enough $\alpha$. The paper is strictly devoted to the 
	Fröhlich polaron, and many calculations are specific to the precise form of $w(t, x) = \e{-t}/|x|$, in particular 
	for writing it as a mixture of Gaussian measures. Also, the central limit theorem is proved as a 
	'diagonal limit', where the diffusive scaling is applied at the same time as $T$ is sent to infinity. 
	
	We find that the method of \cite{MV19} can be generalized significantly, and thereby becomes both 
	conceptually simpler and more powerful. 
	The central object of our theory is a Gibbs measure on marked partitions of the real 
	line, and it is the infinite volume limit of this measure that needs to be understood in 
	order to study limits and properties of the measure \eqref{pair potential}. Using the mixing properties of the limit 
	measure, we prove a full functional central limit theorem, with the diffusive scaling applied after taking the 
	infinite volume limit. Its proof relies on mixing properties 
	of the limiting Gibbs measure on marked partitions and does not need $\bbP_{\alpha,T}$ to be a 
	mixture of Gaussian measures. 
	Another important result of our paper concerns the conditions under which the central limit theorem holds. 
	\cite{MV19} give a sufficient condition (see Condition  \eqref{goodness of alpha} below) 
	for the validity of their results, which is however not easy to check in general. 
	We show that condition \eqref{goodness of alpha} 
	is equivalent to several natural conditions, one of them being \eqref{Our condition}. In particular, 
	this allows us to treat the Fröhlich polaron for all coupling constants with very little additional effort. 
	We also find that the critical parameter $\lambda$ in 
	\eqref{goodness of alpha} is equal to $\psi(\alpha) - \alpha$, which leads to some intriguing relations 
	between the Gibbs measure on marked partitions, the ground state energy $E_\alpha(0)$ of the quantum polaron, and the 
	overlap $\langle \Omega, \Psi_\alpha \rangle$ between the Fock vacuum $\Omega$ and the ground state $\Psi_\alpha$ 
	of the polaron. We do not explore these relations in depth in the present paper, but we present a few interesting
	calculations in Section \ref{Fun calculations} below.

	The paper is organised as follows: in Section \ref{sec:results}, we present our results and compare them to the 
	results available in the literature. In Section \ref{Section: Outline} we give an outline of the proof
	and discuss where we follow the ideas of \cite{MV19} and where we go beyond them. 
	Section 
	\ref{Alternating processes} is the main technical part of the paper: here we develop a general theory of 
	Gibbs measures on (marked) partitions of the real line and discuss conditions for existence of infinite volume 
	measures. Section \ref{Section: Functional central limit theorem} contains the proof of the central limit theorem.
	In Section \ref{Section: Growth condition for Polaron-type models} we prove a sufficient conditions for 
	the validity of \eqref{Our condition}, and in Section \ref{Fun calculations} 
	we discuss how various quantities for polaron models are connected to the 
	measures from Section \ref{Alternating processes}.   
	 	
\section{Results} 	
\label{sec:results}
	
Consider the Polaron path measure \eqref{pair potential} defined on $(C(\bbR, \bbR^d),\caA)$, 
where $\caA$ is the $\sigma$-algebra generated by the increment maps 
\[
X_{s,t}: C(\R, \R^d)\to \bbR^d, \quad \mathbf x \mapsto  X_{s,t}(\mathbf x) := \mathbf x_{s, t} \coloneqq \mathbf x_t - \mathbf x_s
\]
for $s,t \in \bbR$, and where $\caW$ is the distribution of 
Brownian increments (i.e. the restriction of the distribution of a two-sided Brownian motion to $\mathcal A$). For the function $w$ and the parameter $\alpha$ we make the following assumptions: \\[2mm]
(A1): $w$ is measurable, positive (allowing the value $+\infty$), and  fulfills $w(t, -x) = w(t, x)$ for all $t>0$ and $x\in \R^d$. \\[1mm]
(A2): $Z_{\alpha, T} < \infty$ for all $T \geq 0$.\\[1mm]
(A3): $\limsup_{T \to \infty} \frac{Z_{\alpha, 2T}}{Z_{\alpha, T}^2} <\infty$, i.e.\ \eqref{Our condition} holds.
\begin{remark}
\label{Remark: Potentials that are bounded below}
The positivity condition in (A1) can be relaxed: let
$f:[0, \infty) \to [0, \infty)$ be a measurable function with $\int_0^\infty (1+t) f(t) \, \dd t < \infty$. The measures 
$\bbP_{\alpha,T}$ are invariant under replacing $w$ with  $(t, x) \mapsto w(t, x)+f(t)$ in the double integral, and a 
a short calculation shows that \eqref{Our condition} holds after the transformation if and only if it held before the transformation. Therefore, it is enough to assume that there exists $f$ with the above integrability condition 
such that $w+f$ is positive. In other words, as long as $w$ is bounded below and its negative part decays sufficiently
fast at infinity uniformly in $x$, our results hold. 
The well-definedness condition (A2) is trivial for bounded $w$, but becomes an issue for unbounded $w$ such as in the 
Fröhlich polaron. While Proposition \ref{prop:sufficientConditionsForGC} contains some sufficient conditions under which (A2) holds, we refer to the work of Bley and Thomas \cite{BlTh17} for much more general, 
quantitative estimates on the partition functions of Feynman-Kac type perturbations of Brownian motion, as well as an
overview over the literature on this  topic.
\end{remark}

On  $(C(\R, \R^d), \mathcal A)$ , we say that a family $(\mathcal P_t)_{t>0}$ of probability measures
converges to a probability measure $\mathcal P$  {\em  locally in total variation} if for all $a<b$ the 
restriction of $\mathcal P_t$ to
\begin{equation*}
\mathcal A_a^b \coloneqq \sigma\big(X_{s, t}: \, s, t\in [a, b] \big)
\end{equation*}
converges in total variation to the restriction of $\mathcal P$ to $\mathcal A_a^b$ as $t\to \infty$.
	
\begin{thrm}[Existence of infinite volume measure]
	\label{Limit Theorem}
	Assume (A1)--(A3).
	Then there exists a measure $\mathbb P_{\alpha, \infty}$ on $(C(\R, \R^d), \mathcal A)$ such that $\mathbb P_{\alpha, T} \to \mathbb P_{\alpha, \infty}$ locally in total variation as $T \to \infty$.
	\end{thrm}

There is an explicit representation for $\mathbb P_{\alpha, \infty}$, which however  
requires us to introduce and explain several further notions. We refer to Theorem 
\ref{Theorem: Existence infinite volume measure, long version} below for details. 

For stating the central limit theorem, let 
\[
X_t \coloneqq X_{0,t}, \qquad \text{and} \qquad X^n_t \coloneqq \frac{1}{\sqrt{n}}X_{nt}
\]
for $n\in \N$ and $t\in \R$. Recall that a function $u: \bbR^d \to \bbR \cup \{\infty\}$ is called {\em quasiconcave} if its superlevel sets 
$u^{-1}([a, \infty])$, $a\in \R \cup \{\infty\}$ are convex sets. An example for a quasiconcave function is $x \mapsto w(t, x)$ with $w$ as in the Fröhlich polaron.

	\begin{thrm}[Functional central limit theorem]
	\label{Functional central limit theorem}
	Assume (A1)--(A3), and assume that $w(t, \cdot)$ is quasiconcave for all $t>0$. Then there exists a positive definite matrix $\Sigma\in \R^{d\times d}$ with $\Sigma \leq I_d$ (in the sense of quadratic forms) such that the distribution of $X^n$ under $\mathbb P_{\alpha, \infty}$ converges weakly to the distribution of $\sqrt{\Sigma}X$ under $\mathcal W$ as $n\to \infty$.
	\end{thrm}
	Again, there is an explicit representation for the covariance matrix $\Sigma$ that we do not have the notation and notions to state yet. We refer to  Theorem \ref{Theorem: Functional central limit theorem, long version} below. 
	
We now discuss sufficient conditions for \eqref{Our condition}. 
Let $H:D(H) \subseteq  \mathscr H \to \mathscr H$ be a self-adjoint, lower bounded operator in some Hilbert space $\mathscr H$ with inner product $\langle \cdot , \cdot \rangle$ and let $\Phi \in \mathscr H$. We say $(Z_{\alpha, T})_{T\geq 0}$ {\em is of spectral type with Hamiltonian $H$ and state $\Phi$} if 
\begin{equation*}
	Z_{\alpha, T} = \langle \Phi, \e{-T H} \Phi \rangle
\end{equation*}
for all $T\geq 0$. Clearly, this implies the validity of condition (A2). If the bottom of the spectrum of $H$ is an eigenvalue, any normalized eigenvector corresponding to that eigenvalue is called a {\em ground state} of $H$. In this case, we say that $H$ has a ground state. We denote by $\nu_{H, \Phi}$ the spectral measure of $H$ with respect to $\Phi$, i.e. $\nu_{H, \Phi}(A) = \langle \Phi, \1_{A}(H) \Phi \rangle$ for all $A \in \mathcal B(\R)$.
\begin{prop} 
\label{prop:sufficientConditionsForGC} 
\phantom{linebreak}\\[1mm]
a) Assume that there exists a measurable function $f:[0,\infty) \to [0, \infty)$ satisfying $\int_0^\infty (1+t)f(t)\, \mathrm dt< \infty$ such that $|w(t, x)| \leq f(t)$ for all $t\geq 0$ and $x\in \R^d$. Then assumption (A2) and (A3) are fulfilled. \\[1mm]
b) Assume that (A1) holds and that $(Z_{\alpha, T})_{T\geq 0}$ is of spectral type with Hamiltonian $H$ and state $\Phi$ where $H:D(H) \subseteq  \mathscr H \to \mathscr H$ is a self-adjoint, lower bounded operator and $\Phi \in \mathscr H$. Then 
\begin{equation*}
	\psi(\alpha) = - \inf \operatorname{supp}(\nu_{H, \Phi})
\end{equation*}
and condition (A3) is satisfied if and only if $-\psi(\alpha)$ is an eigenvalue of $H$ and $\Phi$ is non-orthogonal to the respective eigenspace. In particular, if $H$ has a ground state $\Psi$ with $\langle \Phi, \Psi \rangle \neq 0$ then (A3) is satisfied and $-\psi(\alpha)$ is the ground state energy.
\\[1mm]
c) 	Assume that $w(t, \cdot) = \tilde w(t, |\cdot|)$ is rotationally symmetric  and positive and that $\tilde w(t, \cdot)$ is decreasing for all $t\geq 0$.
	Let $\tilde w_\beta(r) \coloneqq \sup_{t \geq 0}e^{\beta t}\tilde w(t, r)$ for $\beta, r\geq 0$.
	Assume there exist $\beta>0$ and $p>1$ such that
	\begin{equation*}
	\begin{cases*}
	\int_0^1 \tilde w_\beta(r)^p \, \mathrm dr < \infty \quad &\text{ in case $d=1$} \\
	\int_0^1 r|\log(r)| \cdot \tilde w_\beta(r)^p \, \mathrm dr < \infty &\text{ in case $d=2$} \\
	\int_0^1 r \cdot \tilde w_\beta(r)^p \, \mathrm dr < \infty &\text{ in case $d\geq 3$}.
	\end{cases*}
	\end{equation*}
	Then Condition (A2) is fulfilled for all $\alpha \geq 0$, 
	 $\psi(\alpha)$ exists and is finite for all $\alpha \geq 0$, and there exists 
	an $\alpha_0>0$ such Condition (A3)  is fulfilled for all $\alpha \leq \alpha_0$.
\end{prop} 

\begin{proof} 
For the proof of a), notice that for all $T> 0$
\[
\int_{-T}^T \int_{-T}^T |w(t-s,X_{s,t})| \, \dd s \, \dd t \leq 
\int_{-T}^T \int_{-T}^T f(|t-s|) \, \dd s \, \dd t \leq
4 T \int_{0}^{\infty} f(\tau) \, \dd \tau < \infty,
\]
and thus (A2) holds. The estimate
	\begin{equation}
	\label{eq: first simple estimate} 
	\int_{-T}^0 \int_0^T |w(t-s, X_{s, t})| \, \mathrm dt \, \mathrm d s \leq \int_{0}^\infty \int_{s}^\infty f(\tau) 
	\,  \mathrm d \tau \, \mathrm ds  = \int_0^\infty \tau f(\tau) \, \mathrm d \tau < \infty
	\end{equation}
	leads to
	\begin{equation}
	\label{eq: second simple estimate}
	Z_{\alpha, 2T} \leq Z_{\alpha, T}^2\exp\bigg(\alpha \int_0^\infty \tau f(\tau) \, \mathrm d \tau \bigg)
	\end{equation}
	for all $\alpha, T>0$. This shows (A3).

For b), let $E \coloneqq \inf \supp(\nu_{H, \Phi})$. Since
\begin{equation*}
Z_{\alpha, T} = e^{-TE} \langle \Phi, e^{-T(H-E)} \Phi \rangle = e^{-TE}\int_{[E,\infty)} \e{- T(x-E)}  \nu_{H, \Phi}(\dd x)
\end{equation*}
holds for all $T\geq 0$, we have
\[
\psi(\alpha) = \lim_{T \to \infty} \frac{1}{T}\log Z_{\alpha,T} = - E + \lim_{T \to \infty} \frac{1}{T} 
\log \int_{[E,\infty)} \e{- T(x-E)}  \nu_{H, \Phi}(\dd x) = - E,
\]
the last equality holding because for each $\varepsilon>0$, we have 
\[
1 \geq \int_{[E,\infty)} \e{- T(x-E)}  \nu_{\Phi}(\dd x) \geq \e{-T\varepsilon } \int_{[E, E + \varepsilon)} \nu_{H, \Phi}(\dd x), 
\]
with the integral on the right hand side being strictly positive. Thus
\begin{equation*}
\lim_{T \to \infty} Z_{\alpha, T} e^{-\psi(\alpha) T} = \lim_{T \to \infty} \int_{[E, \infty)} e^{- T(x-E)} \nu_{H, \Phi}(\dd x) = \nu_{H, \Phi}(\{E\})
\end{equation*}
by dominated convergence. Hence, the limit on the left hand side is positive if and only if $E$ is an eigenvalue of $H$ and $\Phi$ is non-orthogonal to the corresponding eigenspace. If the limit is positive, $\lim_{T \to \infty} Z_{\alpha, 2T}/Z_{\alpha, T}^2 = (\lim_{T \to \infty} Z_{\alpha, T} e^{-\psi(\alpha) T})^{-1}$ and \eqref{Our condition} is satisfied. We will see in Theorem \ref{Theorem: Equvialent conditions for goodness in the setting of path measures} that the other direction holds as well: If \eqref{Our condition} is satisfied then $\lim_{T \to \infty} Z_{\alpha, T} e^{-\psi(\alpha) T}$ exists and is positive.\\
The proof of part c) is more involved and is given in Section \ref{Section: Growth condition for Polaron-type models}
below.
\end{proof}

\begin{cor}
\label{cor:polaron} 
Theorems \ref{Limit Theorem} and \ref{Functional central limit theorem} hold for the Fröhlich polaron, 
for all values of $\alpha \geq 0$. 
\end{cor} 

\begin{proof} 
For the Fröhlich polaron, $(Z_{\alpha, T})_{T \geq 0}$ is of spectral type where $H = H_\alpha(0)$ is the Fröhlich Hamiltonian at total momentum zero, 
and $\Phi = \Omega$ is the Fock vacuum; for instance, set $k=t=0$ in formulae (2.11) and (2.14) of \cite{DySp20} to see 
this. It is well known that for all coupling constants $\alpha > 0$, $H(0)$ has a spectral gap and an unique (up to a phase) ground state 
$\Psi_\alpha$, and that $\langle \Psi_\alpha,\Omega \rangle \neq 0$; again, we refer to \cite{DySp20}. 
We can therefore apply Proposition \ref{prop:sufficientConditionsForGC} b) and obtain the result. 
\end{proof} 

The rather direct connection of Condition \eqref{Our condition} to quantum systems given in  
\ref{prop:sufficientConditionsForGC} b) has some interesting implications: it means that independent criteria for the 
validity of \eqref{Our condition}, such as those from Proposition \ref{prop:sufficientConditionsForGC} c) or 
from Theorem \ref{Theorem: Equvialent conditions for goodness in the setting of path measures} below, 
can be used to decide about the existence of ground states in quantum systems. Used in the other direction, this 
connection indicates that the sufficient condition from  Proposition \ref{prop:sufficientConditionsForGC} a) 
is very close to being also necessary for the case of bounded $w$: 
for the massless Nelson model without infrared cutoff it is known 
that $w$ is bounded with $w(t, 0) \sim |t|^{-2}$, and that the Hamiltonian has no ground state in Fock 
space; see Chapter 6.7 of \cite{LHB11}. Therefore \eqref{Our condition} cannot hold in this case.

Let us end this section by giving a more detailed comparison of our results with those previously available. We 
will take  Remark \ref{Remark: Potentials that are bounded below} into account, i.e. that the assumption of positivity in (A1) can be weakened. Then 
Proposition \ref{prop:sufficientConditionsForGC} a) implies that \eqref{Our condition} holds for all cases 
considered in \cite{BeSp04} as well as those included in Assumption 2.1 of \cite{Mu17}. For the alternative 
Assumption 2.2 of \cite{Mu17}, Proposition \ref{prop:sufficientConditionsForGC} c) guarantees that 
\eqref{Our condition} holds for small enough $\alpha$ in all cases except the singular potential in one dimension. 
Also, note that Assumption 2.2 of \cite{Mu17} does not carry a restriction on the value of $\alpha$. 

For the central limit theorem, we require quasiconcavity of $w$, which is not needed in 
\cite{BeSp04} and in the cases from \cite{Mu17} where $w$ is bounded. Quasiconcavity is used for a short proof 
of both tightness in path space and finite variance, via Gaussian correlation inequalities.  
We believe that there should be other methods of achieving these goals whenever \eqref{Our condition} holds, 
but at the moment we do not have them. 

Gubinelli \cite{Gu06} also requires $w$ to be bounded, but proves the existence of an infinite volume limit (but not 
the central limit theorem) also in some cases where the criterion of Proposition 
\ref{prop:sufficientConditionsForGC} a) does not hold. As discussed above in the context of the massless Nelson 
model, we do not expect \eqref{Our condition} to hold in general in such cases. In this context, the work 
\cite{OsSp99} is also of interest, where in a non-translation invariant version of the model with slow decay of 
$t \mapsto w(x,t)$, the existence of at least two different infinite volume limit measures is proved, 
depending on boundary conditions. Since we believe that (A1) - (A3) should already 
be sufficient for a central limit theorem, and since such a central limit theorem is not expected  
to hold in situations where correlations are so strong that multiple Gibbs measures exist, we conjecture that 
in generic cases of bounded potentials violating the assumptions of Proposition \ref{prop:sufficientConditionsForGC} a), 
\eqref{Our condition} will not hold. 

The results in \cite{MV19} were obtained using a method that is related in some aspects to the proof we will give for 
Proposition \ref{prop:sufficientConditionsForGC} c), although other aspects differ significantly. An important 
common feature is that in both methods, there is a restriction on the range of $\alpha$ that can be treated: 
Proposition \ref{prop:sufficientConditionsForGC} c) only works for small $\alpha$, while the results in 
\cite{MV19} are stated for large enough or small enough $\alpha$. We believe that there is a gap in the 
proof of Theorem 4.1 of \cite{MV19}. Namely, it is unclear to us where the factor of $\frac{\alpha}{\alpha + \lambda}$ 
comes from in formula (4.18) of \cite{MV19}; such a factor 
corresponds to the 'dormant' period at the beginning of a cluster, but if (4.18) were true, 
a dormant period would occur for each point in the cluster, which is not the case. 
Without the presence of that factor, however, the choice (4.19) for $\lambda(\alpha)$ is no longer possible when 
$\alpha$ is large, and this creates a gap in the proof of Theorem 4.8 for the case of large $\alpha$. On the other hand, our Corollary \ref{cor:polaron} shows that the results of 
\cite{MV19} are correct as stated and extends them to all $\alpha \geq 0$. 

In conclusion, the extended method of Mukherjee and Varadhan in conjunction with \eqref{Our condition} is a 
powerful tool to study infinite volume limits and functional central limit theorems for Polaron path 
measures 
including the Fröhlich polaron for all coupling strengths. One 
restriction of the method is that $w$ has to be bounded below; this is necessary for positivity 
properties of the measure on partitions. Consequently, e.g.\ the 'anti-polaron' $w(t, x) = - \e{-|t|} / |x|$, 
can not be treated. Another restriction is the requirement of translation invariance. 
This excludes interesting cases, e.g.\ 
measures corresponding to particles interacting both with a quantized field and an external potential; see 
\cite{Be03,BHLMS02,OsSp99}, for instance.

\section{Outline of the proof}
\label{Section: Outline}
Let $\triangle = \{(s,t)\in \R^2: s<t\}$. We denote by $(\mathbf N(\triangle), \mathcal N(\triangle))$ and $(\mathbf N_f(\triangle), \mathcal N_f(\triangle))$ the state spaces of point processes and finite point processes on $\triangle$ respectively. In the following, we always assume that Assumptions (A1) and (A2) hold. We start by modifying the approach taken by Mukherjee and Varadhan for the Polaron \cite{MV19}: By using the series expansion of the exponential function and exchanging the order of integration, we write the measure $\mathbb P_{\alpha, T}$ as a mixture of probability measures $\mathbf P_\xi$, $\xi \in \mathbf N_f(\triangle)$ which have a product structure under non-intersecting ``clusters''. Contrary to \cite{MV19} (where the mixing measure is the distribution of a point process on $\triangle \times (0, \infty)$), our measures $\mathbf P_\xi$ are in general not Gaussian (not even for the Fröhlich polaron). We discuss the connection between both representations in Remark \ref{Remark: Measures as mixtures of Gaussian measures}.
It turns out to be useful to absorb a time damping factor into our Poisson point process. For this, 
we choose a measurable function $v:[0, \infty) \times \R^d \to (0, \infty]$ and a probability density $g:[0, \infty) \to (0, \infty)$ with finite first moment such that 
\begin{equation}
\label{eq: w decomposition} 
	w(t, x) = g(t) v(t, x)
\end{equation}
for all $t\geq 0$ and $x\in \R^d$ (this representation is clearly not unique). Let $\Gamma_{\alpha, T}$ be the distribution of a Poisson point process on $\triangle$ with intensity measure
\begin{equation*}
\mu_{\alpha, T}(\mathrm ds \mathrm dt) \coloneqq \alpha \cdot g(t-s) \mathds 1_{\{-T < s < t < T\}}\mathrm d s \mathrm dt
\end{equation*}
and $c_{\alpha, T}\coloneqq \mu_{\alpha, T}(\triangle)$. If $\tau_1$ is a radom variable with density $g$, then
\begin{equation}
	\label{growth of measure of intensity measure(triangle)}
	c_{\alpha, T}\coloneqq \mu_{\alpha, T}(\triangle) = 2\alpha T - \alpha \int_0^{2T} \mathbb P(\tau_1 > s) \, \mathrm ds
\end{equation}
(which will be relevant at a later point). We then have for $A\in \mathcal A$
\begin{align}
\label{Equation: Paths measure by PPP}
&\mathbb P_{\alpha, T}(A) \nonumber\\
&=  \frac{1}{Z_{\alpha, 2T}} \int_{A}\mathcal W(\mathrm d \mathbf x) \, \exp \left( \alpha \int_{-T}^T \mathrm ds\, \int_{s}^T \mathrm dt\, g(t-s) v\big(t-s, \mathbf x_{s, t}\big)  \right) \nonumber\\
&= \frac{1}{Z_{\alpha, 2T}} \sum_{n=0}^\infty \frac{1}{n!} \int_{\triangle^n} \mu_{\alpha, T}^{\otimes n}(\mathrm d s_1 \mathrm d t_1, \hdots, \mathrm d s_n \mathrm d t_n) \bigg[\int_A\mathcal W(\mathrm d\mathbf x) \prod_{i=1}^{n}v\big(t_i-s_i, \mathbf x_{s_i, t_i}\big) \bigg]\nonumber \\
&=\frac{e^{c_{\alpha, T}}}{Z_{\alpha, 2T}} \int_{\mathbf N_f(\triangle)} \Gamma_{\alpha, T}(\mathrm d\xi) \int_A \mathcal W( \mathrm d \mathbf x) \prod_{(s, t)\in \operatorname{supp}(\xi)  } v\big(t-s, \mathbf x_{s, t}\big).
\end{align}
Note that already here, we used the symmetry and nonngativity conditions from Condition (A1).
For $\xi = \sum_{i=1}^n \delta_{(s_i, t_i)} \in \mathbf N_f(\triangle)$, 
we define
\begin{equation*}
	 F(\xi) \coloneqq \mathbb E_{\mathcal W} \bigg[ \prod_{i=1}^{n} v(t_i-s_i,X_{s_i, t_i}) \bigg]
\end{equation*}
and the measure $\mathbf P_{\xi}$ on $(C(\R, \R^d), \mathcal A)$ by
\begin{equation*}
\mathbf P_\xi(\mathrm d \mathbf x) \coloneqq \frac{1}{F(\xi)}  \prod_{i=1}^{n} v(t_i-s_i,\mathbf x_{s_i, t_i}) \, \mathcal W( \mathrm d \mathbf x)
\end{equation*}
(where $1/\infty \coloneqq 0$). Additionally, we define the measure $\widehat \Gamma_{\alpha, T}$ on $(\mathbf N_f(\triangle), \mathcal N_f(\triangle))$ by
\begin{equation*}
	\widehat{\Gamma}_{\alpha, T} (\mathrm d \xi) \coloneqq \frac{e^{c_{\alpha, T}}}{Z_{\alpha, 2T}} F(\xi) \, \Gamma_{\alpha, T}(\mathrm d \xi).
\end{equation*}
We then have $F(\xi) < \infty$ for $\Gamma_{\alpha, T}$ almost all $\xi \in \mathbf N_f(\triangle)$ (choose $A = C(\R, \R^d)$ in \eqref{Equation: Paths measure by PPP}) and
\begin{equation}
	\label{Polaron-like measures in terms of PPP}
	\mathbb P_{\alpha, T}(\cdot) = \int_{\mathbf N_f(\triangle)} \widehat{\Gamma}_{\alpha, T}(\mathrm d \xi)\, \mathbf P_{\xi}(\cdot).
\end{equation}
In particular $\mathbf P_\xi$ is a probability measure for $\widehat{\Gamma}_{\alpha, T}$ almost all $\xi \in \mathbf N_f(\triangle)$ and  $\widehat{\Gamma}_{\alpha, T}$ is a probability measure.
We have now interchanged the task of proving convergence of the path measures $\mathbb P_{\alpha, T}$ with proving convergence of the distributions of suitable point processes.
Notice that
\begin{equation}
\label{Equation: partition function and fat partition function}
\mathbf Z_{\alpha, 2T} \coloneqq \int_{\mathbf N_f(\triangle)} \Gamma_{\alpha, T}(\mathrm d\xi)\, F(\xi) = Z_{\alpha, 2T}e^{-c_{\alpha, T}}
\end{equation}
acts as a normalization constant in the definition of the reweighted measure $\widehat{\Gamma}_{\alpha, T}$.
One can view $\xi = \sum_{i=1}^n \delta_{(s_i, t_i)} \in \mathbf N_f(\triangle)$ as a collection of intervals $[s_i, t_i]$, $1\leq i \leq n$. We call $\xi$ a {\em cluster}, if these intervals are overlapping in the sense that $\bigcup_{i=1}^n [s_i, t_i]$ is an interval. For an interval $I\subseteq \R$ we say that $\xi$ {\em falls into }$I$ if $[s_i, t_i] \subseteq I$ for all $1\leq i\leq n$. Due to independence of increments under $\mathcal W$, the measures $\mathbf P_{\xi}$ have a product structure along non-intersecting clusters:
Let $I_1, \hdots, I_n$ be intervals such that for $i \neq j$ the intersection $I_i \cap I_j$ contains at most one point. Let $\xi_1, \hdots, \xi_n \in \mathbf N_f(\triangle)$ be such that $\xi_k$ falls into $I_k$ for all $1\leq k \leq n$ and let $\xi = \sum_{k=1}^n\xi_k$. Then 
\begin{equation*}
	F(\xi) = \prod_{k=1}^{n}F(\xi_k).
\end{equation*}
If we additionally assume that $F(\xi_k)<\infty$ for all $k\in \{1, \hdots, n\}$ then the processes of increments $(X_{s, t})_{s, t\in I_1}, \hdots, (X_{s, t})_{s, t\in I_n}$ are independent under $\mathbf P_{\xi}$ and for all $k \in \{1, \hdots, n\}$
\begin{equation*}
	\mathbf P_\xi \circ \Big((X_{s, t})_{s, t\in I_k}\Big)^{-1} = \mathbf P_{\xi_k} \circ \Big((X_{s, t})_{s, t\in I_k}\Big)^{-1}.
\end{equation*}	

In particular, for $a<b$, $\xi \in \mathbf N_f(\triangle)$ and $A\in \mathcal A_a^b$, the probability $\mathbf P_{\xi}(A)$ only depends on all clusters of $\xi$ that intersect $[a, b]$.
We denote by $R_a^b$ the restriction to all clusters that intersect $[a, b]$, that is for $\xi = \sum_{i \in I} \delta_{(s_i, t_i)}\in \mathbf N(\triangle)$
\begin{equation*}
	R_a^b \xi \coloneqq \sum_{i\in J_a^b(\xi)} \delta_{(s_i, t_i)} \text{ with } J_a^b(\xi) \coloneqq \max \big\{J\subseteq I: \, \bigcup_{i \in J} [s_i, t_i] \cup [a, b] \text{ is an interval} \big\}.
\end{equation*}
Then for all $A\in \mathcal A_a^b$
\begin{equation}
\label{Mixing of restriction of path measures}
\mathbb P_{\alpha, T}(A) = \int_{\mathbf N_f(\triangle)} \widehat{\Gamma}_{\alpha, T}(\mathrm d\xi)\, \mathbf P_{R_a^b \xi}(A).
\end{equation}
In Section \ref{Alternating processes}, we prove Theorem \ref{Limit Theorem} i.e. the existence of an infinite volume measure by showing that for all $a<b$ the measures $\widehat{\Gamma}_{\alpha, T} \circ (R_a^b)^{-1}$ converge in total variation as $T\to \infty$ provided that \eqref{Our condition} holds. We show:
\begin{prop}
	\label{local covergence hatGamma}
	Assume that Conditions (A1)-(A3) hold. 
	Then there exists a stationary (with respect to translations along the diagonal) measure $\widehat{\Gamma}_{\alpha, \operatorname{st}}$ on $(\mathbf N(\triangle), \mathcal N(\triangle))$ such that for all $a<b$ we have $\widehat \Gamma_{\alpha, T} \circ (R_a^b)^{-1} \to \widehat \Gamma_{\alpha, \operatorname{st}} \circ (R_a^b)^{-1}$ in total variation as $T\to \infty$.
\end{prop}
We give a short summary of the argument used. 
Consider a Poisson point process $\eta = \sum_{i=1}^\infty \delta_{(s_i, t_i)}$ with $s_1 < s_2 < \hdots$ on $\triangle$ with intensity measure $\alpha g(t-s) \1_{\{-T<s<t\}} \, \mathrm ds \mathrm dt$. Here $s_i$ and $t_i$ can be interpreted at the arrival and departure of the $i$-th customer of a $M/G/\infty$-queue (Poisson arrivals, general service time distribution, infinitely many servers) started empty at time $-T$ with arrival intensity $\alpha$ and service time distribution with density $g$. Conditionally on the event that no customer is present at $T$, the process of all customers arriving between $-T$ and $T$ has distribution $\Gamma_{\alpha, T}$. We can decompose the $M/G/\infty$-queue into successive busy cycles (a dormant period, in which no customers are present, followed by an active period in which at least one customer is present). We denote by $\xi_1$ the first cluster of the queue, i.e. the process of all customers arriving during the first busy cycle, shifted in time such that the first arrival is at time zero. We denote by $T_1$ the length of the first busy cycle, i.e. the sum of the first dormant and the first active period. As identified by Mukherjee and Varadhan in \cite{MV19} for the Polaron measure, the condition
\begin{equation}
\label{goodness of alpha}
	\text{There exists a } \lambda \in \R \text{ such that } \mathbb E_{\alpha} \big[e^{-\lambda T_1} F(\xi_1)\big] = 1 \text{ and }\mathbb E_{\alpha} \big[T_1 e^{-\lambda T_1 } F(\xi_1)\big] < \infty \tag{G}
\end{equation}
is sufficient to conclude Proposition \ref{local covergence hatGamma}. In Section \ref{Alternating processes} we show this in much greater generality, not needing any specific properties of the distribution $\Xi_\alpha$. The argument is as follows: The condition $\mathbb E_{\alpha} \big[e^{-\lambda T_1} F(\xi_1)\big] = 1$ guarantees that $\widehat{\Xi}_{\alpha}(\mathrm d \xi) \coloneqq \tfrac{\alpha}{\alpha + \lambda} e^{-\lambda a(\xi)} F(\xi) \Xi_\alpha(\mathrm d \xi)$ defines  a probability measure (where $a(\xi)$ denotes the length of the active period of $\xi$). As the sum of all active and dormant periods intersecting $[-T, T]$ is $2T$ (up to a correction term that is irrelevant due to memorylessness) we can multiply $F$ with factors $e^{-\lambda a_i}$ and $e^{-\lambda d_j}$ for each active period $a_i$ and each dormant period $d_j$ intersecting $[-T, T]$ without changing the Gibbs measure $\widehat{\Gamma}_{\alpha, T}$. As $F$ is multiplicative in the clusters, the measure $\widehat \Gamma_{\alpha, T}$ can hence be obtained by simply reweighting the distribution of exponentially distributed dormant periods and $\Xi_\alpha$ distributed clusters: Starting in $-T$ we alternate independently drawn $\operatorname{Exp}(\alpha + \lambda)$ distributed dormant periods with $\widehat{\Xi}_\alpha$ distributed clusters. Conditionally on the event that no customer is present at $T$, the process of customers arriving between $-T$ and $T$ has distribution $\widehat{\Gamma}_{\alpha,T}$. The second condition $\mathbb E_{\alpha} \big[T_1 e^{-\lambda T_1 } F(\xi_1)\big] < \infty$ yields the finiteness of the expected value of the cluster length for the reweighted cluster distribution. This allows us to apply renewal theory to deduce the existence of a stationary version $\widehat \Gamma_{\alpha, \text{st}}$ of the process obtained by alternating the reweighted distributions of dormant periods and active clusters. Due to the memorylessness property of the exponential distribution, the measure $\widehat \Gamma_{\alpha, T}$ can alternatively be seen as the process of all customers under $\widehat \Gamma_{\alpha, \text{st}}$ that arrive between $-T$ and $T$ conditionally on the boundary condition that the system is dormant at $-T$ and $T$. Unsurprisingly, locally the effect of the boundary condition vanishes as $T \to \infty$.
Condition \eqref{goodness of alpha} is rather difficult to verify directly, but we find that it is equivalent to 
several natural conditions, including \eqref{Our condition}.
Remember that we denote by $\psi(\alpha) \coloneqq \lim_{T \to \infty} \log(Z_{\alpha, T})/T \in \R \cup \{\infty\}$ the free energy (the limit exists by the version of Feketes lemma for measurable superadditive functions). We set $e^{-\infty} \coloneqq 0$. In Section \ref{Alternating processes} we show:
\begin{thrm}
	\label{Theorem: Equvialent conditions for goodness in the setting of path measures}
	Assume (A1) and (A2) hold. Then the following conditions are equivalent:
	\begin{enumerate}
		\item There exists a $\lambda \in \R$ such that $\mathbb E_{\alpha} \big[e^{-\lambda T_1} F(\xi_1)\big] = 1$ and $\mathbb E_{\alpha} \big[T_1 e^{-\lambda T_1 } F(\xi_1)\big] < \infty$
		\item $\limsup_{T\to \infty} Z_{\alpha, 2T}/Z_{\alpha, T}^2 < \infty$
		\item $\liminf_{T\to \infty} \widehat \Gamma_{\alpha, T}( \text{the system is dormant at }0) > 0$
		\item $\liminf_{T \to \infty} Z_{\alpha, T} e^{-\psi(\alpha)T}>0$.
	\end{enumerate}
	If (1)-(4) hold then $\lambda = \psi(\alpha) - \alpha$ is the unique real number satisfying \eqref{goodness of alpha} and
	\begin{equation*}
	\lim_{T \to \infty} Z_{\alpha, T} e^{-\psi(\alpha)T} = \frac{1}{\psi(\alpha) \mathbb E_\alpha[ T_1 e^{-(\psi(\alpha) - \alpha)T_1} F(\xi_1)]}.
	\end{equation*}	
\end{thrm}
We prove Theorem \ref{Theorem: Equvialent conditions for goodness in the setting of path measures} through another application of renewal theory. More specifically, we use that the function $f:[0, \infty) \to \R$ defined by
\begin{equation*}
f(T) \coloneqq e^{-(\psi(\alpha) - \alpha)T} \mathbf Z_{\alpha, T} \cdot \mathbb P_\alpha(\text {the }M/G/\infty \text{-queue is dormant at $T$})
\end{equation*}
satisfies the renewal equation
\begin{equation*}
f(T) = \mathbb E\big[\1_{\{T_1 \leq t\}} e^{-(\psi(\alpha) - \alpha) T_1} F(\xi_1) f(T-T_1)\big] + e^{-\psi(\alpha) T} \quad \text{ for all }T\geq 0.
\end{equation*}
Combining the previous considerations, we obtain 
\begin{thrm}[Existence of infinite volume measure]
	\label{Theorem: Existence infinite volume measure, long version}
	Assume that
	(A1)-(A3) hold.
	Then there exists a measure $\mathbb P_{\alpha, \infty}$ on $(C(\R, \R^d), \mathcal A)$ such $\mathbb P_{\alpha, T} \to \mathbb P_{\alpha, \infty}$ locally in total variation as $T\to \infty$. For any $a<b$ the restriction of $\mathbb P_{\alpha, \infty}$ to $\mathcal A_{a}^b$ is given by
	\begin{equation*}
	\mathbb P_{\alpha, \infty}(A) = \int_{\mathbf N(\triangle)} \widehat \Gamma_{\alpha, \operatorname{st}}(\mathrm d \xi) \, \mathbf P_{R_a^b\xi}(A)
	\end{equation*}
	for all $A \in \mathcal A_{a}^b$ where $\widehat \Gamma_{\alpha, \operatorname{st}}$ is the stationary distribution of the process obtained by alternating $\operatorname{Exp}(\psi(\alpha))$ distributed dormant periods and the distribution
	\begin{equation*}
	\widehat \Xi_\alpha(\mathrm d\xi) = \frac{\alpha}{\psi(\alpha)} e^{-(\psi(\alpha)-\alpha) a(\xi)} F(\xi) \, \Xi_\alpha(\mathrm d\xi)
	\end{equation*}
	of active clusters.
\end{thrm}

For the Fröhlich Polaron path measure, Mukherjee and Varadhan \cite{MV19} show (for $\alpha$ such that \eqref{goodness of alpha} holds) an ordinary central limit theorem in the ``diagonal limit'', using that their counterparts $\mathbf P_{\xi, u}$ of our measures $\mathbf P_\xi$ are Gaussian. We use a different approach that allows us to show Theorem \ref{Functional central limit theorem}, i.e. a full functional central limit theorem, for more general potentials $w$.
Roughly speaking, we use the following rather natural arguments (assuming that \eqref{Our condition} is satisfied):
\begin{itemize}
	\item Let $\hat \eta_s \sim \widehat{\Gamma}_{\alpha, \operatorname{st}}$ be a stationary version of the process obtained by alternating the reweighted dormant periods and clusters. Let $s_1 < t_1 < \hdots < s_k< t_k$. For large $n$ the processes $R_{ns_1}^{nt_1}\hat \eta_s,\hdots, R_{ns_k}^{nt_k}\hat \eta_s$ are approximately independent due to the renewal structure. By the product structure of the measures $\mathbf P_\xi$ along non-intersecting clusters, the increments $X_{ns_1, nt_1}, \hdots, X_{ns_k, nt_k}$ are approximately independent under $\mathbb P_{\alpha, \infty}$ for large $n$.
	\item For $a<b$ and large $n$, approximately $n(b-a)/\hat{\mathbb E}_\alpha[\hat T_1]$ renewal points fall into $[na, nb]$. By the product structure, the increment $\frac{1}{\sqrt{n}}X_{an, bn}$ is under $\mathbb P_{\alpha, \infty}$ approximately the sum of $n(b-a)/\hat{\mathbb E}_\alpha[\hat T_1]$ independent increments.
	\item To prove an ordinary central limit theorem, we still need that the variance of an increment along a single renewal period is finite. For the functional central limit theorem, we additionally need tightness of the family of probability measures $\{\mathbb P_{\alpha, \infty} \circ (X^n)^{-1}: \, n\in \N\}$. We  assume that $w(t, \cdot)$ is quasiconcave for all $t>0$ and apply the Gaussian correlation inequality in order to show these two prerequisites.
\end{itemize}
While the assumption of quasiconcavity of $w(t, \cdot)$ is satisfied for the Fröhlich polaron and several other cases, it is conceivable that at least the finite variance along a single renewal period (and thus the ordinary central limit theorem) can be obtained under weaker assumptions on $w$. We do not pursue this any further here and settle for:
\begin{thrm}[Functional Central limit theorem]
	\label{Theorem: Functional central limit theorem, long version}
	Assume that (A1)-(A3) hold, and that $w(t, \cdot)$ is quasiconcave for all $t>0$.
	Then the distribution of $X^n$ under $\mathbb P _{\alpha, \infty}$ converges weakly to the distribution of $\sqrt{\Sigma} X$ under $\mathcal W$ where
	\begin{equation}
	\label{Equation: Covariance matrix}
	\Sigma \coloneqq  \frac{\hat{\mathbb E}_\alpha[\Sigma(\hat d_1, \hat \xi_1)]}{\hat{\mathbb E}_\alpha[\widehat T_1]} \quad \text{with } \Sigma(\hat d_1, \hat \xi_1)  \coloneqq \hat d_1 I_d+ \mathbf E_{\hat \xi_1}\big[ X_{0, \hat a_1} \cdot X_{0, \hat a_1} ^T \big],
	\end{equation}
	and $(\hat d_1, \hat \xi_1) \sim \operatorname{Exp}(\psi(\alpha)) \otimes \widehat \Xi_{\alpha}$, $\hat a_1$ is the length of the cluster $\hat \xi_1$ and $\widehat T_1 = \hat a_1 + \hat d_1$.
	Additionally, $\Sigma \leq I_d$ in the sense of quadratic forms.
\end{thrm}
Here $\hat d_1 I_d$ is the contribution of a dormant period on which the process is a Brownian motion. Equation \eqref{Equation: Covariance matrix} directly entails that $\Sigma$ is positive definite. In order to show that $\Sigma \leq I_d$, we use quasiconcavity of $w(t, \cdot)$ and the Gaussian correlation inequality.\\
\\
If we replace $w$ with the potential $(t, x) \mapsto w(t, x) + g(t)$ our partition function becomes $Z_{\alpha, 2T} e^{c_{\alpha, T}}$. By Equation \eqref{growth of measure of intensity measure(triangle)}, 
\begin{equation}
\label{growth of calphat}
e^{c_{\alpha, T}} \sim e^{2\alpha T -\alpha \mathbb E[\tau_1]}
\end{equation}
as $T\to \infty$. Thus, adding $g$ to $w$ does not change whether \eqref{Our condition} holds or not. After the transformation of the potential, we have $v>1$ and thus $\mathbb E_\alpha[F(\xi_1)]>1$, which turns out to be beneficial for showing \eqref{goodness of alpha} for sufficiently small $\alpha$: Then, by the intermediate value theorem, the condition $\mathbb E_{\alpha}[F(\xi_1)] < \infty$ is sufficient such that \eqref{goodness of alpha} is satisfied. We apply this in Section \ref{Section: Growth condition for Polaron-type models} and prove Proposition \ref{prop:sufficientConditionsForGC} c). Besides giving a similar, slightly stronger,  upper estimate on $F$ (providing necessarily a different proof which does not rely on the specific choice of $w$) our approach is different from that used in \cite{MV19}, replacing the Markov chain argument with an application of the optional stopping theorem, and thus using an argument that does not need $g$ to be the density of an exponential distribution.
At the end of Section \ref{Section: Growth condition for Polaron-type models} (see Section \ref{Fun calculations}) we will use the fact that $\psi(\alpha) - \alpha$ is the unique solution to \eqref{goodness of alpha} (assuming that \eqref{Our condition} holds) and perform a few heuristic calculations that connect the free energy to properties of the reweighted point process.

\section{Gibbs measures on partitions of $[-T, T]$}
\label{Alternating processes}
It will be notationally convenient to identify a point process on $\triangle$ that alternates between dormant periods and active clusters with a random partition of $\R$ into dormant and active intervals, each active interval beeing equipped with the corresponding cluster as a ``state''.
We will consider general systems which alternate between a dormant state and (possible multiple) active states. Starting dormant at time zero, the system switches into an active state after an exponentially distributed waiting time. After an independently drawn waiting time (whose distribution might depend on the state) the system turns dormant again. Iterating this procedure generates a partition of $[0, \infty)$, where each interval of the partition is in a certain state. For a visualization see Figure \ref{fig:partitionrealaxis}. By an application of renewal theory, there exists a stationary version of the process. We will use its distribution as a reference measure for a certain class of Gibbs measures under the dormant boundary condition. We will assume that dormant intervals do not contribute energy and that the Hamiltonian is additive in the active states. Given a growth condition on the partition function, we will see that the Gibbs measures can be obtained by simply reweighting the distributions of the independent components mentioned above. This yields the existence of an infinite volume measure. \\
\\
Before continuing, we give a short summary of some standard results from renewal theory. A renewal process $(S_n)_{n\geq 0}$ is a sequence of random variables such that the so called interarrival times $(T_n)_{n\geq 1} \coloneqq (S_{n} - S_{n-1})_{n \geq 1}$ form an iid sequence of a.s. positive random variables and such that $S_0 = 0$ almost surely. For $t\geq 0$ let $N(t) \coloneqq \#\{n: \, S_n \leq t\}$ be the number of renewal points in the interval $[0, t]$. The renewal function $U: [0, \infty) \to [0, \infty)$,  $U(t) \coloneqq \mathbb E[N(t)]$ is the expected number of renewal points (including the origin) in the interval $[0, t]$. Then it holds
\begin{thrm}\cite[p. 140]{asmussen2003applied}
	\label{Elementary Renewal Theorem}
	Irrespective whether $\mathbb E[T_1]$ is finite or infinite (setting $1/\infty = 0$) we have as $t\to \infty$
	\begin{equation*}
		N(t)/t \to 1/\mathbb E[T_1] \text { almost surely and }U(t)/t \to 1/\mathbb E[T_1].
	\end{equation*}
\end{thrm}
Let $\mu$ be a Radon measure on $[0, \infty)$ with $\mu(\{0\}) = 0$. Let $z:[0, \infty)\to \R$ be a locally bounded measurable function. Then the renewal equation
\begin{equation*}
Z(t) = \int_{[0, t]} Z(t-s) \, \mu(\mathrm ds) + z(t) \quad \text{for all }t\geq 0
\end{equation*}
has an unique locally bounded measurable solution $Z_z: [0, \infty) \to \R$ given by the convolution of $z$ with the renewal measure $\sum_{n=0}^\infty \mu^{*n}$. We will apply the renewal theorem in the following form:
\begin{thrm}[Renewal theorem]
	\label{renewal theorem}\cite{Arjas1978}
	If $\mu$ is a probability measure that is absolutely continuous with respect to the Lebesgue measure and $g:[0, \infty) \to [0, \infty)$ is bounded and Lebesgue integrable with $g(t) \to 0$ as $t\to \infty$ then
	\begin{equation*}
	\lim_{t \to \infty} \sup_{|z|\leq g} \Big|Z_z(t) - \frac{1}{\int_0^\infty s \,  \mu(\mathrm ds)} \int_0^\infty z(s) \, \mathrm ds \Big|  = 0.
	\end{equation*}	
\end{thrm}
For sub-probability measures, the following holds:
\begin{prop}\cite[p. 163]{asmussen2003applied}
	Assume $\mu([0, \infty))<1$. If $z$ is bounded and $\lim_{t \to \infty} z(t)$ exists then
	\begin{equation*}
	\label{renewal for subprobability-measures}
	\lim_{t \to \infty} Z_z(t) = \frac{\lim_{t \to \infty} z(t)}{1 - \mu([0, \infty))}.
	\end{equation*}	
\end{prop}
The forward recurrence time $A_t$ of the renewal process $(S_n)_n$ at time $t$ is the timespan from $t$ up to the next renewal after $t$, i.e.
\begin{equation*}
	A_t \coloneqq  \inf \{S_n - t:\, n\in \N \text{ s.t. } S_n>t\}.
\end{equation*}
If $T_1$ has a density and $\mathbb E[T_1] < \infty$, one can apply the renewal theorem in order to show that the distribution of $A_t$ converges in total variation as $t\to \infty$. 
\\
\\
Let $\mathcal X$ be a Polish space which will serve as our state space. Let $\alpha>0$ and $\Xi$ be a probability distribution on $(0, \infty) \times \mathcal X$.
Let $(d_n)_n$ be an iid sequence of $\operatorname{Exp}(\alpha)$ distributed random variables and $((a_n, \chi_{n}))_n$ be an iid sequence of $\Xi$ distributed random variables, independent of $(d_n)_n$. We will always assume that $\Xi$ is such that $\mathbb E[a_1] <\infty$. Here $d_n$ should be interpreted as the $n$-th dormant period and $a_n$ as the $n$-th active period with state $\chi_{n}$. If, for example, we look at the process that represents the $M/G/\infty$-queue (started empty at time $0$) we choose $\mathcal X = \mathbf N_f(\triangle)$ and $\chi_{n} = \xi_n$ to be the process of customers arriving during the $n$-th busy period (shifted in time such that the first customer arrives at time 0). For $n\in \N_0$, we define
\begin{equation*}
X_{2n} \coloneqq \sum_{k=1}^{n} d_k + a_k, \quad \quad
X_{2n+1}\coloneqq d_{n+1} + \sum_{k=1}^{n} d_k + a_k,
\end{equation*}
consider the point process
\begin{equation*}
\eta \coloneqq \sum_{n=1}^{\infty} \delta_{(X_{2n-1}, X_{2n}, \chi_{n})}
\end{equation*}
on $\triangle \times \mathcal X$, and denote its distribution by $\Gamma^\Xi_\alpha$. The point process shall be interpreted as ``on the interval $[X_{2n-1}, X_{2n})$ the system is in state $\chi_{n}$ and on the interval $[X_{2n}, X_{2n+1})$ the system is dormant''. We call the renewal process $(S_n)_n \coloneqq (X_{2n})_n$ the embedded renewal process of $\eta$, see Figure \ref{fig:partitionrealaxis} for a visualization. For $n\in \N$ we denote by $T_n \coloneqq d_n + a_n$ the $n$-th interarrival time of embedded renewal process.
\begin{figure}
	\centering
	\includegraphics[width=0.8\linewidth]{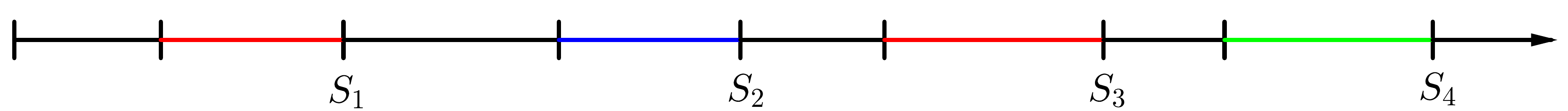}
	\caption{Example of a partition with statespace $\mathcal X = \{\textcolor{red}{\text{red}}, \textcolor{blue}{\text{blue}}, \textcolor{green}{\text{green}}\}$. Dormant periods are in black.}	
	\label{fig:partitionrealaxis}
\end{figure}
For $t>0$ let
\begin{equation*}
	\tau_t \coloneqq \inf \{n:\,S_n \geq t\}
\end{equation*}
be the hitting time of $[t, \infty)$ of the embedded renewal process of $\eta$. Notice that $\tau_t$ coincides with the number of renewal points in $[0, t)$ (counting the origin as a renewal point). The distribution of $T_1$ is absolutely continuous with respect to the Lebesgue measure as $a_1$ and $d_1\sim \operatorname{Exp}(\alpha)$ are independent. Since
\begin{equation*}
\mathbb P(\eta \text{ is dormant at }t) = \mathbb E\big[ \1_{\{T_1 \leq t\}}  \mathbb P(\eta \text{ is dormant at }t-T_1) \big] + \mathbb P(d_1 > t)
\end{equation*}
for all $t\geq 0$ the renewal theorem implies
\begin{equation*}
\lim_{t \to \infty} \mathbb P(\eta \text{ is dormant at }t) = \mathbb E[d_1]/\mathbb E[T_1].
\end{equation*}For $a\leq b$ let $M_{a, b} \coloneqq \{(s, t, x) \in \triangle \times \mathcal X: \, [s, t] \cap [a, b] \neq \emptyset\}$ and let
\begin{equation*}
R_a^b : \mathbf N(\triangle \times \mathcal X) \to \mathbf N(\triangle \times \mathcal X), \quad \mu \mapsto \mu(\, \cdot \, \cap M_{a, b})
\end{equation*}
be the restriction to all marked intervals that intersect the interval $[a, b]$.
Let us consider again the example of the process $\eta$ representing the $M/G/\infty$-queue. Then $R_a^b \eta$ contains not only all individuals that arrive or depart during $[a, b]$ but all clusters that intersect $[a, b]$. In a similar manner, we define $R_0^\infty$ to be the restriction to all marked intervals intersecting $[0, \infty)$.
For $t>0$ we define 
\begin{equation*}
\theta_t:\mathbf N(\triangle \times \mathcal X) \to \mathbf N(\triangle \times \mathcal X), \quad \sum_{i\in I} \delta_{(s_i, t_i, x_i)} \mapsto \sum_{i\in I} \delta_{(s_i-t, t_i-t, x_i)}
\end{equation*}
to be the translation by $-t$. We call a probability measure $\mathcal P$ on $(\mathbf N(\triangle \times \mathcal X), \mathcal N(\triangle \times \mathcal X))$ stationary, if $\mathcal P = \mathcal P \circ \theta_t^{-1}$ for all $t \in \R$. We call a point process on $\triangle \times \mathcal X$ stationary if its distribution is stationary.
By an application of renewal theory, we obtain the existence of a stationary version $\Gamma^{\Xi}_{\alpha, \operatorname{st}}$ of $\Gamma^\Xi_\alpha$.

\begin{prop}
	\label{convergence to stationarity}
	There exists a stationary point process $\eta^s$ on $\triangle\times \mathcal X$ such that the distribution of $R_0^\infty \theta_t \eta$ converges in total variation to the distribution of $R_0^\infty\eta^s$ as $t\to \infty$. Additionally, we have
	\begin{equation*}
	R_0^\infty\eta^s \overset{d}{=} \gamma + \theta_{-T_0} \eta
	\end{equation*}
	where $(T_0, \gamma)$ is a $(0, \infty)\times \mathbf N_f(\triangle \times \mathcal X)$ valued random variable independent of $\eta$ with 
	\begin{equation}
	\label{limit starting distribution}
		\mathbb P\big((T_0, \gamma) \in C\big) = \frac{1}{\mathbb E[T_1]} \int_{0}^\infty \mathbb P\big((T_1 - t, \theta_t R_t^{T_1} \eta) \in C\big) \, \mathrm dt
	\end{equation}
	for all $C\in \mathcal B((0, \infty)) \otimes \mathcal N_f(\triangle \times \mathcal X)$.
\end{prop}
\begin{proof}
	First, we show convergence of the distribution of $R_0^\infty \theta_t \eta$ as $t\to \infty$. The proof follows the general idea: ``Condition with respect to a suitable chosen initial condition and show convergence of the distribution of the initial condition'', see e.g. \cite[p.224 ff.]{Daley2007} \\
	If $A_t$ is the forward recurrence time of the embedded renewal process at $t$ and $\gamma_t \coloneqq \theta_t R_t^{t+A_t} \eta$
	is the re-shifted marked interval ending in $t+A_t$
	then we have for all $B\in \mathcal N(\triangle \times \mathcal X)$
	\begin{equation*}
	\mathbb P(R_0^\infty \theta_t \eta \in B) = \int_{(0, \infty)\times \mathbf N_f(\triangle \times \mathcal X)} \mathbb P( \zeta + \theta_{-s}\eta \in B)\, \big( \mathbb P \circ (A_t, \gamma_t)^{-1} \big)(\mathrm ds\,  \mathrm d\zeta).
	\end{equation*}
	Now for all $C\in \mathcal B((0, \infty)) \otimes \mathcal N_f(\triangle \times \mathcal X)$ and $t\geq 0$
	\begin{equation*}
		\mathbb P( (A_t, \gamma_t) \in C) = \mathbb E\big[ \1_{\{T_1 \leq t\}} \mathbb P\big((A_{t-T_1}, \gamma_{t-T_1}) \in C\big)\big] + \mathbb P\big((T_1 - t, \theta_t R_t^{T_1} \eta) \in C\big).
	\end{equation*}
	By the renewal theorem (with the majorant $g(t) \coloneqq \mathbb P(T_1>t)$ for $t\geq 0$), the distribution of $(A_t, \gamma_t)$ converges in total variation to the distribution defined by the right hand side of \eqref{limit starting distribution} as $t\to \infty$. This implies the convergence of the distribution of $R_0^\infty \theta_t \eta$ in total variation as $t\to \infty$. By construction, the limit $\mathcal P$ satisifes $\mathcal P \circ (R_0^\infty \theta_t)^{-1} = \mathcal P$ for all $t>0$ and can be extended to a stationary distribution on $\mathbf N(\triangle \times \mathcal X)$.	
\end{proof}
We denote by $\Gamma_{\alpha, \operatorname{st}}^\Xi$ the distribution of $\eta^s$.
Below we will use $\Gamma_{\alpha, \operatorname{st}}^\Xi$ as a reference measure for certain Gibbs measures. For $t>0$ we denote by $\Gamma^\Xi_{\alpha, t}$ the distribution of $R_{-t}^t\theta_t \eta$ conditionally on $\{\eta \text{ is dormant at $2t$} \}$.
Notice that, by the memorylessness property of the exponential distribution, $\Gamma^\Xi_{\alpha, t}$ is also the distribution of the process $R_{-t}^t\eta^s$ conditionally on $\{\eta^s \text{ is dormant at $-t$ and $t$}\}$. Unsurprisingly, locally the effect of the boundary condition vanishes as $t\to \infty$.
	Given probability measures $(\mathcal P_t)_t$, $\mathcal P$ on $\big(\mathbf N(\triangle \times \mathcal X), \mathcal N(\triangle \times \mathcal X)\big)$ we say $\mathcal P_t \to \mathcal P$ {\em  locally in total variation} if, for any $a<b$, we have $\mathcal P_t \circ (R_a^b)^{-1} \to \mathcal P\circ (R_a^b)^{-1} $ in total variation as $t\to \infty$.	
\begin{prop}
	\label{Proposition: local convergence of alternating process}
	We have $\Gamma^{\Xi}_{\alpha, t} \to \Gamma_{\alpha, \operatorname{st}}^\Xi$ locally in total variation as $t \to \infty$.
\end{prop}
\begin{proof}
	Let $a<b$. To shorten notation, we denote for $t>0$ by $\sigma_t \coloneqq \tau_{t+b}$ the hitting time of $[t+b, \infty)$ of the embedded renewal process of $\eta$ and set
	\begin{equation*}
	N_t(\eta) \coloneqq \begin{cases}
	0  \quad &\text{if $\eta$ is dormant at $t$} \\
	1  \quad &\text{if $\eta$ is active at $t$}.
	\end{cases}
	\end{equation*}
	As a consequence of the renewal theorem, we have
	\begin{equation}
	\label{convergence probability dormant}
	\mathbb P(N_t(\eta) = 0) \to \mathbb E[d_1]/\mathbb E[T_1] \eqqcolon c
	\end{equation}
	as $t\to \infty$. Let $\varepsilon>0$ and let $t_1 >0$ be such that for all $t \geq t_1$
	\begin{equation}
	\label{Twoestimates}
	\mathbb P(N_{2t}(\eta) = 0) \geq c/2, \quad |\mathbb P(N_{2t}(\eta) = 0) - c| \leq \varepsilon.
	\end{equation} 
	For $t\geq 0$ let $s_t \coloneqq b + t + \sqrt{t}$.
	Let $t_2 \geq 0$ be such that for all $t\geq t_2$ we have $ t -b  - \sqrt{t}\geq t_1$. For $t\geq t_2$ we then have on $\{S_{\sigma_t} \leq s_t\}$ that $2t - S_{\sigma_t} \geq t_1$ holds. By conditioning with respect to the process up the first renewal after $b+t$, we get for all $A\in \mathcal N(\triangle \times \mathcal X)$ and $t\geq t_2$
	\begin{align*}
	&\big|\mathbb P(R_a^b \theta_{t} \eta\in A, S_{\sigma_t} \leq s_t , N_{2t}(\eta) = 0) - \mathbb P(R_a^b \theta_{t} \eta\in A, S_{\sigma_t} \leq s_t) \mathbb P(N_{2t}(\eta) = 0)\big| \\
	&\leq \int_\Omega \mathbb P( \mathrm d\omega)\,\Big[\mathds 1_{\{S_{\sigma_t} \leq s_t\}}(\omega) \big|\mathbb P(N_{2t-S_{\sigma_t}(\omega)}(\eta) = 0) - \mathbb P(N_{2t}(\eta) = 0)\big| \,\Big] \\
	&\leq 2 \varepsilon.
	\end{align*}
	On $\{S_{\sigma_t} > s_t\}$ the forward recurrence time at $t+b$ is at least $\sqrt{t}$. By convergence of the distribution of the forward recurrence time there exists a $t_3\geq t_2$ such that for all $t\geq t_3$
	\begin{equation*}
	\mathbb P(S_{\sigma_t} > s_t) \leq \varepsilon.
	\end{equation*}
	We then have for all $t \geq t_3$ 
	\begin{align*}
	&\big|\mathbb P(R_a^b \theta_{t} \eta \in A, N_{2t}(\eta) = 0) - \mathbb P(R_a^b \theta_{t} \eta \in A) \mathbb P(N_{2t}( \eta) = 0)\big| \\
	&\leq \mathbb P(R_a^b \theta_{t} \eta \in A, S_{\sigma_t} > s_t, N_{2t}(\eta) = 0) \\
	&+ |\mathbb P(R_a^b \theta_{t} \eta \in A, S_{\sigma_t} \leq s_t, N_{2t}(\eta) = 0) - \mathbb P(R_a^b \theta_{t} \eta \in A, S_{\sigma_t} \leq s_t) \cdot \mathbb P(N_{2t}(\eta) = 0)| \\
	&+ \mathbb P(R_a^b \theta_{t} \eta \in A, S_{\sigma_t} > s_t) \cdot \mathbb P(N_{2t}(\eta) = 0) \\
	& \leq 4 \varepsilon
	\end{align*}
	and hence
	\begin{equation*}
	|\mathbb P(R_a^b \theta_{t} \eta \in A|N_{2t}(\eta) = 0) - \mathbb P(R_a^b \theta_{t} \eta \in A)| \leq \frac{4 \varepsilon}{\mathbb P(N_{2t}(\eta) = 0)} \leq \frac{4 \varepsilon}{c/2}.
	\end{equation*}
	Combining this with Proposition \ref{convergence to stationarity} yields the claim.	
\end{proof}

We will now consider Gibbs measures with respect to the stationary distribution $\Gamma_{\alpha, \text{st}}^\Xi$ under dormant boundary conditions. We will assume that dormant intervals have no energy contribution and that the Hamiltonian is additive in the active states. Let $E: \mathcal X \to \R \cup \{-\infty\}$ be measurable. We define $\mathcal H: \mathbf N_f(\triangle \times \mathcal X) \to \R \cup \{-\infty\}$ by
\begin{equation*}
\mathcal H\Big( \sum_{k=1}^n \delta_{(s_k, t_k, x_k)}\Big) \coloneqq \sum_{k=1}^n E(x_k)\quad \text{for }  \sum_{k=1}^n \delta_{(s_k, t_k, x_k)} \in \mathbf N_f(\triangle \times \mathcal X).
\end{equation*}
In the following, we set $e^{\infty} \coloneqq \infty \eqqcolon \log(\infty)$.
\begin{prop}
	\label{Proposition: Partition function locally bounded}
	Assume that there exists a $\mu \in \R$ such that $\mathbb E[e^{-\mu T_1 - E(\chi_1)}] < \infty$. Then
	\begin{equation*}
		t \mapsto \mathbf Z_{2t} \coloneqq \int_{\mathbf N_f(\triangle \times \mathcal X)} e^{-\mathcal H(\zeta)}  \, \Gamma^{\Xi}_{\alpha, t}(\mathrm d\zeta)
	\end{equation*}
	is locally bounded, more precisely there exists a $\lambda \in \R$ and a $C>0$ such that $\mathbf Z_t \leq C e^{\lambda t}$ for all $t\geq 0$.
\end{prop}
\begin{proof}
	By the dominated convergence theorem, there exists a $\lambda\in \R$ such that
	\begin{equation*}
		c \coloneqq \mathbb E[e^{-\lambda T_1 - E(\chi_1)}] \leq 1.
	\end{equation*}
	If $A_t$ denotes for $t\geq0$ the forward recurrence time of the embedded renewal process of $\eta$ at time $t$ we have
	$T_1 + \hdots +T_{\tau_t} = t + A_t$. 
	By the memorylessness property of the exponential distribution, $(A_t, \chi_{\tau_t})$ is independent of $R_0^t \eta$ conditionally on $\{\eta \text{ is dormant at }t\}$ and
	we have for all $t\geq 0$
	\begin{align*}
		\mathbf Z_{t}e^{- \lambda t} &= \frac{1}{c\cdot \mathbb P(\eta \text{ is dormant at }t)} \mathbb E\Big[\prod_{i=1}^{\tau_t} e^{-\lambda T_i - E(\chi_i)} \1_{\{\eta \text{ is dormant at }t\}} \Big] \\
		&\leq \frac{1}{c \cdot \mathbb P(\eta \text{ is dormant at }t)} \mathbb E\Big[\prod_{i=1}^{\tau_t} e^{-\lambda T_i - E(\chi_i)} \Big].
	\end{align*}
	As $(\prod_{i=1}^{n} e^{-\lambda T_i - E(\chi_i)})_{n\geq 1}$ is a supermartingale with respect to the filtration generated by $((d_n, a_n, \chi_n))_{n \geq 1}$ we get by an application of the optional stopping theorem and Fatou's lemma
	\begin{equation*}
		 \mathbb E\Big[\prod_{i=1}^{\tau_t} e^{-\lambda T_i - E(\chi_i)} \Big] \leq \liminf_{n \to \infty} \mathbb E\Big[\prod_{i=1}^{\tau_t\wedge n} e^{-\lambda T_i - E(\chi_i)} \Big]  \leq 1.
	\end{equation*}
	 The claim follows as $\lim_{t\to \infty}\mathbb P(\eta \text{ is dormant at }t) >0$.
\end{proof}

From now on, we will always assume that $t \mapsto \mathbf Z_{t}$
is locally bounded. For $t>0$ we define the probability measure $\widehat{\Gamma}^{\Xi}_{\alpha, t}$ by
\begin{equation*}
\widehat{\Gamma}^{\Xi}_{\alpha, t}(\mathrm d \zeta) = \frac{1}{\mathbf Z_{2t}}e^{-\mathcal H(\zeta)}\, \Gamma^{\Xi}_{\alpha, t}(\mathrm d \zeta).
\end{equation*}
To make contact with the mixing measure in the representation introduced in Section \ref{Section: Outline}, consider $\mathcal X = \mathbf N_f(\triangle)$ and $(d_n)_n$, $(a_n)_n$  to be the sequences of dormant and active periods of the $M/G/\infty$ queue and $(\chi_{n})_n = (\xi_n)_n$ to be the sequence of clusters (i.e. the reshifted processes of customers arriving during the respective active periods) and $E = -\log(F)$ with
\begin{equation*}
	F\Big(\sum_{i=1}^n \delta_{(s_i, t_i)}\Big) \coloneqq \mathbb E_\mathcal W\Big[\prod_{i=1}^{n}v(t_i - s_i, X_{s_i, t_i})\Big].
\end{equation*}

We give a condition under which the reweighted measure $\widehat{\Gamma}^{\Xi}_{\alpha, t}$ can be obtained by simply reweighting the distributions of the dormant periods $(d_k)_k$ and active periods and states $(a_k, \chi_{k})_k$. If the reweighted active periods have finite expectation, this will in particular imply local convergence in total variation of the measures $\widehat{\Gamma}^{\Xi}_{\alpha, t}$ as $t\to \infty$.
\begin{prop}
	\label{Proposition: Tilting measures}
	Assume $\lambda \in \R$ satisfies
	\begin{equation*}
	\mathbb E\big[e^{-\lambda T_1 - E(\chi_1)}\big] = 1.
	\end{equation*}
	Then 
	\begin{equation*}
	\widehat{\Gamma}^{\Xi}_{\alpha, t} = \Gamma^{\widehat{\Xi}}_{\alpha + \lambda, t}  \quad \text{ where } \quad \widehat{\Xi}(\mathrm d a \, \mathrm d \chi) \coloneqq \frac{\alpha}{\alpha + \lambda} e^{-\lambda a - E(\chi)} \, \Xi(\mathrm d a \, \mathrm d \chi)
	\end{equation*}
	for all $t>0$. Additionally, 
	\begin{equation*}
		\mathbf Z_{t} = e^{\lambda t} \cdot \frac{\Gamma_{\alpha}^{\widehat{\Xi}}(\text{the system is dormant at $t$})}{\Gamma_{\alpha}^\Xi(\text{the system is dormant at $t$})}.
	\end{equation*}
\end{prop}
\begin{proof}
	First, notice that $\widehat \Xi$ indeed defines a probability measure as
	\begin{equation*}
		1 = \mathbb E\big[e^{-\lambda T_1 -E(\chi_1)}\big] = \mathbb E\big[e^{-\lambda d_1}\big] \mathbb E\big[e^{-\lambda a_1 -E(\chi_1)}\big] = \frac{\alpha}{\alpha + \lambda} \mathbb E\big[e^{-\lambda a_1 -E(\chi_1)}\big].
	\end{equation*}
	Let $(\hat d_n)_n$ be an iid sequence of $\operatorname{Exp}(\alpha + \lambda)$ distributed random variables. Let $((\hat a_n, \hat \chi_{n}))_n$ be an iid sequence of $\widehat{\Xi}$ distributed random variables, independent of $(\hat d_n)_n$, and $\hat \eta$ be the process constructed in the same manner as in the beginning of the section. As usual, we denote for $t\in \R$ by $\hat \tau_t$ and $\tau_t$ the hitting times of $[t, \infty)$ of the embedded renewal processes of $\hat \eta$ and $\eta$ respectively. For $n\in \N_0$, define $\delta_n \coloneqq 2t - S_n$ and $\hat \delta_n \coloneqq 2t - \hat S_n$. Then
	\begin{equation*}
	\{\tau_{2t} = n+1,\, \eta \text{ is dormant at $2t$}\} = \{S_n\leq 2t, d_{n+1} > \delta_n\}.
	\end{equation*}
	The distribution of $\hat Y = (\hat d_1, \hdots, \hat d_{n+1}, \hat a_1, \hat \chi_1, \hdots, \hat a_n, \hat \chi_n)$ has a density with respect to the distribution of $Y = (d_1, \hdots, d_{n+1}, a_1, \chi_1, \hdots, a_n, \chi_n)$ given by 
	\begin{equation*}
		(\tilde d_1, \hdots, \tilde d_{n+1}, \tilde a_1, \tilde \chi_1, \hdots, \tilde a_n, \tilde \chi_n) \mapsto \frac{\alpha + \lambda}{\alpha} e^{-\lambda  \tilde d_{n+1}} \prod_{i=1}^n e^{-\lambda (\tilde d_i + \tilde a_i)}e^{-E(\tilde \chi_i)}.
	\end{equation*}
	For measurable $f:\mathbf N_f(\triangle \times \mathcal X) \to [0, \infty)$ we have that $f(R_{-t}^t \theta_t \hat\eta)\1_{ \{\hat S_n\leq 2t, \hat d_{n+1} > \delta_n\}} = h(\hat Y)$ for a suitable measurable function $h$. By writing the following expected value as an integral with respect to the image measure of $\hat Y$ and using that
	\begin{equation*}
	e^{-\lambda(T_1 + \hdots  + T_n + d_{n+1}) } = e^{-2\lambda t - \lambda (d_{n+1} - \delta_n)}
	\end{equation*}	
	one obtains
	\begin{align*}
	&\mathbb E\big[f(R_{-t}^t \theta_t \hat\eta)\1_{ \{\hat S_n\leq 2t,\,\hat d_{n+1} > \hat \delta_n\}}\big]\\
	&= \tfrac{\alpha + \lambda}{\alpha} e^{-2\lambda t} \mathbb E\big[f(R_{-t}^t \theta_t \eta)\1_{ \{S_n\leq 2t, d_{n+1} > \delta_n\}} e^{-\sum_{i=1}^{n}E(\chi_{i})}  e^{- \lambda (d_{n+1} - \delta_n)}\big].
	\end{align*}
	By the memorylessness property of the exponential distribution, we have for all $t\geq 0$
	\begin{equation*}
		\mathbb E\big[e^{-\lambda(d_{n+1} - t)} \1_{\{d_{n+1} > t\}}\big] = \frac{\alpha}{\alpha + \lambda} \mathbb E\big[\1_{\{d_{n+1} > t\}}\big].
	\end{equation*}
	By conditioning with respect to $(d_1, \hdots, d_n, a_1, \chi_1, \hdots, a_n, \chi_n)$ we hence obtain
	\begin{equation*}
		\mathbb E\big[f(R_{-t}^t \theta_t \hat\eta)\1_{ \{\hat S_n\leq 2t,\,\hat d_{n+1} > \hat \delta_n\}}\big] = e^{-2\lambda t} \mathbb E\big[f(R_{-t}^t \theta_t \eta)\1_{ \{S_n\leq 2t, d_{n+1} > \delta_n\}} e^{-\mathcal H(R_0^{2t} \eta)}\big].
	\end{equation*}
	By summing over all $n\in \N_0$ we get 
	\begin{equation*}
		\mathbb E\big[f(R_{-t}^t \theta_t \hat\eta)\1_{ \{\hat \eta \text{ is dormant at }2t  \} }\big] = e^{-2\lambda t} \mathbb E\big[f(R_{-t}^t \theta_t \eta)\1_{ \{ \eta \text{ is dormant at }2t  \}} e^{-\mathcal H(R_0^{2t} \eta)}\big]
	\end{equation*}	
	and hence
	\begin{equation*}
		\int_{\mathbf N_f(\triangle \times \mathcal X)} f(\zeta) \, \Gamma_{\alpha + \lambda, t}^{\widehat \Xi}(\mathrm d \zeta) = e^{-2 \lambda t} \frac{\mathbb P(\eta \text{ is dormant at }2t)}{\mathbb P(\hat \eta \text{ is dormant at }2t)} \int_{\mathbf N_f(\triangle \times \mathcal X)} f(\zeta) e^{-\mathcal H(\zeta)} \, \Gamma^{\Xi}_{\alpha, t}(\mathrm d \zeta)
	\end{equation*}
	which yields the claims.	
\end{proof}

As a corollary of Proposition \ref{Proposition: Tilting measures} and Proposition \ref{Proposition: local convergence of alternating process} we obtain our first main result, namely a sufficient condition for the existence of an infinite volume measure for the Gibbs measures $\widehat \Gamma^\Xi_{\alpha, T}$. This implies the existence of the infinite volume measure for our path measures $\mathbb P_{\alpha, T}$. 
\begin{cor}
	\label{Corollary: Convergence of GIbbs measures if we are good}
	Assume $\lambda \in \R$ satisfies
	\begin{equation}
	\label{goodness}
	\mathbb E\big[e^{-\lambda T_1 - E(\chi_1)}\big] = 1 \quad \text{ and } \quad \mathbb E\big[T_1 e^{-\lambda T_1 - E(\chi_1)}\big] < \infty. \tag{G'}
	\end{equation}
	Then $\widehat{\Gamma}^{\Xi}_{\alpha, t} \to \Gamma_{\alpha+\lambda, \operatorname{st}}^{\widehat{\Xi}} \eqqcolon \widehat \Gamma_{\alpha, \operatorname{st}}^\Xi$ locally in total variation.
\end{cor}

Note that Condition \eqref{goodness} specializes to \eqref{goodness of alpha} in the context of Section 
\ref{Section: Outline}.
 We will now show that the existence of a real number $\lambda$ satisfying \eqref{goodness} is equivalent to several natural conditions, including exponential growth of the partition function with $t$. 

\begin{prop}
	\label{Existence ground state energy}
	The limit $\varphi \coloneqq \lim_{t \to \infty} \log(\mathbf Z_t)/t$ exists in $\R \cup \{\infty\}$.
\end{prop}
\begin{proof}
	For $t>0$ let
	\begin{equation*}
	f(t) \coloneqq \mathbb E[ e^{- \mathcal H(R_0^{t} \eta)} \1_{\{\eta \text{ is dormant at } t\}}] = \mathbf Z_t \mathbb P(\eta \text{ is dormant at } t).
	\end{equation*}
	By the memorylessness property of the exponential distribution, conditionally on the event $\{\eta \text{ is dormant at } t\}$ the process $R_0^{t} \eta$ is independent of $R_{t}^{t+s} \eta$ for all $s, t>0$ and
	\begin{equation*}
		\mathbb P(\theta_t R_t^{t+s}\eta \in\, \cdot \, | \eta \text{ is dormant at }t) = \mathbb P(R_0^s \eta \in \, \cdot \,).
	\end{equation*}
	This yields for all $s, t>0$
	\begin{equation*}
	f(t+s) =  f(t)f(s) +  \mathbb E\big[ e^{ - \mathcal H(R_0^{t +s} \eta)} \1_{\{\eta \text{ is dormant at } t+s\}} \1_{\{\eta \text{ is active at } t\}}\big]
	\end{equation*}
	which implies, by the version of Feketes lemma for measurable superadditive functions, $\log(f(t))/t \to \sup_{s > 0 } \log(f(s))/s$ as $t \to \infty$. Since
		$\lim_{t \to \infty} \mathbb P(\eta \text{ is dormant at } t)>0$, the claim follows.
\end{proof}

\begin{lemma}
	\label{Lemma: limsup of quotient}
	Let $f:(0, \infty) \to (0, \infty)$ be a function such that $a \coloneqq \lim_{t\to \infty} \log(f(t))/t$ exists in $\R$ and such that $\limsup_{t\to \infty}f(t)e^{-at}<\infty$. Then we have
	\begin{equation*}
	\liminf_{t\to \infty} f(t)e^{-a t} > 0 \iff \limsup_{t\to \infty} \frac{f(2t)}{f(t)^2} < \infty.
	\end{equation*}	
\end{lemma}

\begin{proof}
	Define $g(t) \coloneqq f(t) e^{-at}$ for $t>0$.
	First, assume that $\liminf_{t\to \infty} g(t) > 0$. Then 
	\begin{equation*}
	\limsup_{t \to \infty}\frac{f(2t)}{f(t)^2} = \limsup_{t \to \infty} \frac{g(2t)}{g(t)^2} < \infty.
	\end{equation*}
	Now, assume that $\limsup_{t\to \infty} f(2t)/f(t)^2 <\infty$.	Then there exists a $c>1$ and a $T>0$ such that $g(2t) \leq c g(t)^2$  for all $t \geq T$.
	Assume that $\liminf_{t\to \infty} g(t) = 0$ would hold. Then, there would exists a $t_0 \geq T$ such that $cg(t_0) < 1$. Set $t_k \coloneqq 2^kt_0$ for $k\in \N$. Then we have with $b\coloneqq (c g(t_0))^{1/t_0}$
	\begin{equation*}
	g(t_k) \leq c^{2^{k}-1} g(t_0)^{2^k} = \frac{1}{c} (c g(t_0))^{t_k/t_0} < b^{t_k}
	\end{equation*}
	for all $k\in \N$. However, as $\log(g(t))/t \to 0$ as $t\to \infty$ and $\log(b)< 0$ there exists a $k\in \N$ such that $g(t_k) > e^{t_k \log(b)} = b^{t_k}$.
\end{proof}

\begin{thrm}
	\label{Goodness criterions for alternating processes}
	Recall that $\varphi \coloneqq \lim_{t\to \infty} \log(\mathbf Z_t)/t \in \R \cup \{\infty\}$.	The following conditions are equivalent:
	\begin{enumerate}
		\item \label{Condition 1} Condition \eqref{goodness} holds, i.e.\ there exists a $\lambda \in \R$ such that 
		 \[ 
		 \mathbb E\big[e^{-\lambda T_1 - E(\chi_1)}\big] = 1 \qquad \text{and} \qquad 
		 \mathbb E\big[T_1 e^{-\lambda  T_1 - E(\chi_1)}\big] < \infty.
		 \]
		\item \label{Condition 2} $\limsup_{t\to \infty} \mathbf Z_{2t}/\mathbf Z_t^2 < \infty$
		\item  \label{Condition 3} $\liminf_{t\to \infty} \widehat \Gamma_{\alpha, t}^{\Xi}( \text{the system is dormant at }0) > 0$
		\item \label{Condition 4} $\liminf_{t \to \infty} \mathbf Z_t e^{-\varphi t}>0$.
	\end{enumerate}
	If \eqref{Condition 1}-\eqref{Condition 4} hold, then $\lambda = \varphi < \infty$ is the unique real number satisfying Condition \ref{Condition 1} and
	\begin{equation*}
		\lim_{t \to \infty} \mathbf Z_{t} e^{-\varphi t} = \frac{\alpha}{\alpha + \lambda}\frac{\mathbb E[T_1]}{\hat{\mathbb E}[\widehat T_1]}.
	\end{equation*}
\end{thrm}

\begin{proof}
	We will use the same notation as in the proof of Proposition \ref{Proposition: local convergence of alternating process}.
	First, we will point out that Conditions \eqref{Condition 2} and \eqref{Condition 3} are equivalent. Using the same argument as in the proof of Proposition \ref{Existence ground state energy} we obtain
	\begin{align*}
	\widehat \Gamma_{\alpha, t}^\Xi(N_0 = 0) &= \frac{1}{\mathbf Z_{2t} \mathbb P (N_{2t}(\eta) = 0)} \mathbb E[e^{ -\mathcal H(R_0^{t}\eta)} \1_{\{N_{t}(\eta) = 0\}} e^{-\mathcal H(R_{t}^{2t}\eta)} \1_{\{N_{2t}(\eta) = 0\}}]\\
	&= \frac{1}{\mathbf Z_{2t} \mathbb P (N_{2t}(\eta) = 0)} \mathbb E[e^{-\mathcal H(R_0^{t}\eta)} \1_{\{N_{t}(\eta) = 0\}}]^2 \\
	&= \frac{\mathbb P (N_{t}(\eta) = 0)^2}{\mathbb P (N_{2t}(\eta) = 0)} \frac{\mathbf Z_{t}^2}{\mathbf Z_{2t}}.
	\end{align*}
	The equivalence of both conditions follows from
	\begin{equation*}
	\lim_{t \to \infty} \frac{\mathbb P (N_{t}(\eta) = 0)^2}{\mathbb P (N_{2t}(\eta) = 0)} = \frac{\mathbb E[d_1]}{\mathbb E[T_1]}.
	\end{equation*}	
	Similar to the proof of Proposition \ref{Lemma: limsup of quotient}, one can see that Condition \eqref{Condition 2} implies $\varphi<\infty$.
	The equivalence of Conditions \eqref{Condition 2} and \eqref{Condition 4} follows from Lemma \ref{Lemma: limsup of quotient} since
	\begin{equation}
	\label{upper bound on f}
	\log(\mathbf Z_t \mathbb P(N_{t}(\eta) = 0))/t \leq \varphi
	\end{equation}
	holds by superadditivity (as seen in in the proof of Proposition \ref{Existence ground state energy}).
	Finally, we show that Conditions \eqref{Condition 1} and \eqref{Condition 4} are equivalent. Assume that there exists a $\lambda \in \R$ as in Condition \eqref{Condition 1}. By Proposition \ref{Proposition: Tilting measures}, we then have
	\begin{equation*}
	\mathbf Z_t e^{-\lambda t} = \frac{\hat{\mathbb P}(N_{t}(\hat \eta) = 0)}{\mathbb P(N_{t}(\eta) = 0)} \to \frac{ \hat{\mathbb E}_\alpha[\hat d_1] / \hat{\mathbb E}_\alpha[\hat T_1]}{\mathbb E_\alpha[d_1] /\mathbb E_\alpha[T_1]} = \frac{\alpha}{\alpha+ \lambda} \frac{\mathbb E[T_1]}{\hat{\mathbb E}[\hat T_1]}
	\end{equation*}
	as $t\to \infty$. As the limit is finite and positive, this additionally implies that $\log(\mathbf Z_t) -t\lambda$ converges to a real number as $t\to \infty$ and hence, $\lambda = \varphi < \infty$. For the other direction, assume that
	\begin{equation*}
		\liminf_{t \to \infty} \mathbf Z_t e^{-\varphi t}>0.
	\end{equation*}
	Then in particular $\varphi<\infty$. We define $f  :[0, \infty) \to (0, \infty)$ by 
	\begin{align*}
	f(t) \coloneqq e^{-\varphi t} \mathbb E[e^{-\mathcal H(R_{0}^{t}\eta)} \1_{\{N_{t}(\eta) = 0\}}] = \mathbf Z_t e^{-\varphi t} \mathbb P(N_{t}(\eta)=0).
	\end{align*}
	By the usual renewal argument
	\begin{align}
	\label{renewal equation}
	f(t) &= \mathbb E[\1_{\{T_1 \leq t\}} e^{-\varphi T_1-E(\chi_1)} f(t-T_1)] + e^{-\varphi t}\mathbb P(d_1 > t) \\
	&= \mathbb E[\1_{\{T_1  \leq t\}} e^{-\varphi T_1-E(\chi_1)} f(t-T_1)] + e^{-(\varphi + \alpha) t}. \nonumber
	\end{align}
	That is, $f$ satisfies a renewal equation with respect to the image measure of $T_1$ under the (possibly non probability-) measure $ e^{-\varphi T_1-E(\chi_1)} \mathrm d \mathbb P$. Notice that Equation \eqref{upper bound on f} implies $f(t) \leq 1$ for all $t\geq0$ and thus $\varphi> -\alpha$.
	Equation \eqref{renewal equation} in combination with Fatou's Lemma implies
	\begin{equation}
	\label{inequality}
	\mathbb E\Big[\liminf_{t \to \infty}\1_{\{T_1  \leq t\}} e^{-\varphi T_1-E(\chi_1)} f(t - T_1) \Big] \leq \liminf_{t \to \infty} f(t).
	\end{equation}
	Since
	\begin{equation*}
	\liminf_{t\to \infty} f(t) = \liminf _{t\to \infty} \mathbf Z_t e^{-\varphi t} \mathbb P(N_{t}(\eta)=0) = \tfrac{\mathbb E[d_1]}{\mathbb E[T_1]} \liminf _{t\to \infty} \mathbf Z_t e^{-\varphi t} >0
	\end{equation*}
	we can divide both sides of Inequality \eqref{inequality} by $\liminf_{t \to \infty} f(t)$ and obtain
	\begin{equation*}
	\mathbb E[e^{-\varphi T_1-E(\chi_1)}] \leq 1.
	\end{equation*}
	If $\mathbb E[e^{-\varphi T_1-E(\chi_1)}] <1$ or $\mathbb E[T_1 e^{-\varphi T_1-E(\chi_1)}]=\infty$ would hold, then $f$ (as the solution to the renewal equation \eqref{renewal equation}) would converge to zero by renewal theory (compare Proposition \ref{renewal for subprobability-measures}, Theorem \ref{renewal theorem}). We thus conclude that Condition \eqref{Condition 1} holds.
\end{proof}

We finish our treatment of the abstract alternating processes by first considering the case that \eqref{goodness} is not satisfied and then giving two sufficient criteria for \eqref{goodness}.

\begin{prop}
	\label{at least <= 1}
	Assume $\varphi< \infty$. Irrespective whether \eqref{goodness} is satisfied, we have
	\begin{equation*}
		\varphi = \min\Big\{\lambda\in \R: \, \mathbb E \big[e^{-\lambda T_1 - E(\chi_1)}  \big] \leq 1   \Big\}.
	\end{equation*}
	If \eqref{goodness} does not hold then $\lim_{t\to \infty} \mathbf Z_te^{-\varphi t} = 0$. 
\end{prop}
\begin{proof}
	Assume that $\mathbb E[e^{-\varphi T_1 - E(\chi_1)}] > 1$. Then there would exist a $N\in \N$ such that
	\begin{equation*}
		\mathbb E\big[e^{- \varphi T_1 - E(\chi_1)}\1_{\{T_1 \leq N, \, E(\chi_1) \geq -N\}}\big] > 1.
	\end{equation*}	
	By the intermediate value theorem, there would exists a $\lambda>0$ such that
	\begin{equation}
		\label{euqation: phi is minimum}
		\mathbb E\big[e^{-(\lambda + \varphi) T_1 - E(\chi_1)}\1_{\{T_1 \leq N, \, E(\chi_1) \geq -N\}}\big] = 1.
	\end{equation}	
	For $t\geq 0$ set
	\begin{equation*}
		f_N(t) \coloneqq e^{-(\varphi + \lambda) t} \mathbb E\Big[\prod_{i=1}^{\tau_t -1} e^{-E(\chi_{i})} \1_{\{T_i \leq N, \, E(\chi_{i}) \geq -N\}} \1_{ \{N_t(\eta) = 0 \} } \Big].
	\end{equation*}
	As in the proof of Theorem \ref{Goodness criterions for alternating processes}, one can derive a renewal equation for $f_N$. 
	Equation \eqref{euqation: phi is minimum} in combination with the renewal theorem would imply convergence of $f_N(t)$ to a positive real number as $t\to\infty$. This, however, contradicts
	\begin{equation*}
		f_N(t) \leq e^{-(\varphi + \lambda) t} \mathbb E\Big[\prod_{i=1}^{\tau_t -1} e^{-E(\chi_{i})} \1_{ \{N_t(\eta) = 0 \} } \Big] = e^{-(\varphi + \lambda) t} \mathbf Z_t \mathbb P(\eta \text{ is dormant at }t)
	\end{equation*}
	for all $t>0$. Thus  $\mathbb E[e^{-\varphi T_1 - E(\chi_1)}] \leq 1$. If $\mathbb E\big[e^{-\varphi T_1 - E(\chi_1)}\big] < 1$ or $\mathbb E\big[T_1 e^{-\varphi  T_1 - E(\chi_1)}\big] = \infty$ holds then $f$	
	as defined in the proof of Theorem \ref{Goodness criterions for alternating processes} converges to zero and thus
	\begin{equation*}
		\lim_{t \to \infty} \mathbf Z_t e^{-\varphi t} = 0.
	\end{equation*}
	Similiarily it can be shown, that $\varphi$ is indeed the minimum: Assume there would exists a $\varepsilon \in (0, \varphi + \alpha)$ such that
	\begin{equation*}
		\mathbb E[ e^{-(\varphi - \varepsilon)T_1 - E(\chi_1)}] < 1.
	\end{equation*}
	Then, by Proposition \ref{renewal for subprobability-measures}, we would have $e^{\varepsilon t} f(t) \to 0$ as $t\to \infty$ i.e. exponential decay of $t\mapsto \mathbf Z_t e^{-\varphi t}$.	
\end{proof}
\begin{prop}
	\label{Remark: stronger version of goodness}
	The following two conditions are sufficient for \eqref{goodness} to hold:
	\begin{enumerate}
		\item There exists a $\mu \in \R$ such that
		\begin{equation*}
		1< \mathbb E\big[e^{-\mu T_1 - E(\chi_1)}\big] < \infty.
		\end{equation*}
		\item We have $\varphi< \infty$ and there exists an $\varepsilon>0$ such that
		\begin{equation*}
			\mathbb E\big[e^{-(\varphi - \varepsilon) T_1 - E(\chi_1)}\big] < \infty.
		\end{equation*}
	\end{enumerate}
\end{prop}
\begin{proof}
	Assume there exists a $\mu \in \R$ such that $1< \mathbb E\big[e^{-\mu T_1 - E(\chi_1)}\big] < \infty$.
	Then, by the intermediate value theorem, there exists a $\lambda> \mu$ such that $\mathbb E\big[e^{-\lambda T_1 - E(\chi_1)}\big]  = 1$. Then $\mathbb E\big[ T_1 e^{-\lambda T_1 - E(\chi_1)}\big]  < \infty$ holds as well, as $t \in \mathcal O(e^{(\lambda - \mu) t})$ for $t \to \infty$. 
	The second claim follows from the first one and Proposition \ref{at least <= 1}.
\end{proof}

We are now in the position to prove Theorem \ref{Theorem: Equvialent conditions for goodness in the setting of path measures} and Theorem \ref{Theorem: Existence infinite volume measure, long version} by applying the previous results to the measures introduced in Section \ref{Section: Outline}. We consider $\mathcal X = \mathbf N_f(\triangle)$ and $(d_n)_n$, $(a_n)_n$ to be the sequences of dormant and active periods of the $M/G/\infty$-queue and $(\chi_{n})_n = (\xi_n)_n$ to be the sequence of clusters of the queue. Additionally, we choose $E = -\log(F)$. Since the first active period $a_1 \eqqcolon a(\xi_1)$ is a function of the first cluster $\xi_1$, we reduce notation and denote by $\Xi_\alpha$ (as in Section \ref{Section: Outline}) the distribution of $\xi_1$ (and not the joint distribution of $(a_1, \xi_1)$). 

\begin{proof}[Proof of Theorem \ref{Theorem: Equvialent conditions for goodness in the setting of path measures} and Theorem \ref{Theorem: Existence infinite volume measure, long version}]
	The number of customers present at time $T>0$ in the $M/G/\infty$-queue (started empty at time zero) is $\operatorname{Poi}(\beta(\alpha, T))$ distributed, where
	\begin{equation*}
		\beta(\alpha, T) = \alpha \int_0^T \int_T^\infty g(t-s) \, \mathrm dt \mathrm ds  = \alpha \int_0^T \mathbb P(\tau_1 > r) \, \mathrm dr.
	\end{equation*}
	Hence,
	\begin{equation*}
		\mathbb E_\alpha[T_1] = \frac{\mathbb E_\alpha[d_1]}{\lim_{T \to \infty} \mathbb P_\alpha(\eta \text{ is dormant at }T)} = \frac{1/\alpha}{\lim_{T \to \infty} e^{-\beta(\alpha, T)}} = \frac{e^{\alpha \mathbb E[\tau_1]}}{\alpha}.
	\end{equation*}
	As mentioned before
	\begin{equation*}
	e^{c_{\alpha, T}} \sim e^{2\alpha T - \alpha \mathbb E[\tau_1]} = (\alpha \mathbb E[T_1])^{-1} e^{2\alpha T}
	\end{equation*}
	as $T \to \infty$. In particular, since
	\begin{equation*}
	\mathbf Z_{\alpha, 2T} \coloneqq \int_{\mathbf N_f(\triangle)} \mathrm d \Gamma_{\alpha, T}(\xi)\, F(\xi) = Z_{\alpha, 2T}e^{-c_{\alpha, T}}
	\end{equation*}
	we have
	\begin{equation*}
	\varphi(\alpha) \coloneqq \lim_{T \to \infty} \log(\mathbf Z_{\alpha, T})/T = \psi(\alpha) - \alpha.
	\end{equation*}
	In combination with Theorem \ref{Goodness criterions for alternating processes} the above yields Theorem \ref{Theorem: Equvialent conditions for goodness in the setting of path measures}.
	We have
	\begin{equation*}
	\mathbb E_{\alpha}\Big[e^{-(\psi(\alpha)-\alpha) T_1} F(\xi_1) \Big] = 1, \quad \mathbb E_{\alpha}\Big[T_1e^{-(\psi(\alpha)-\alpha) T_1} F(\xi_1) \Big] < \infty
	\end{equation*}
	and thus Corollary \ref{Corollary: Convergence of GIbbs measures if we are good} and Equation \eqref{Mixing of restriction of path measures} yield for all $a<b$
	\begin{equation*}
	\mathbb P_{\alpha, T}(A) \to \int_{\mathbf N(\triangle)} \widehat{\Gamma}_{\alpha, \operatorname{st}}(\mathrm d \xi) \mathbf P_{R_a^b\xi}(A) \eqqcolon \mathbb P_{\alpha, \infty}^{a, b}(A)
	\end{equation*}
	uniformly in $A\in \mathcal A_a^b$ as $T\to \infty$. Consistency of the family of measures $(\mathbb P_{\alpha, \infty}^{a, b})_{a<b}$ implies the existence of the measure $\mathbb P_{\alpha, \infty}$.	
\end{proof}

\begin{remark}
	\label{Remark: Measures as mixtures of Gaussian measures}
	Let us assume that $w(t, \cdot) = \tilde w(t, |\cdot |)$ is rotationally symmetric and positive and that $r \mapsto \tilde w(t, \sqrt{r})$ is completely monotone on $(0, \infty)$ with $\tilde w(t, 0) = \lim_{r \to 0} \tilde w(t, r)$ for all $t>0$. We will represent the measure $\mathbb P_{\alpha, T}$ as a mixture of Gaussian measures. For the Polaron measure, this will make contact between our representation and the representation introduced in \cite{MV19}. Additionally, this representation can be used in order to show that (under these stronger assumptions on $w$) the convergence of the finite dimensional distributions in the proof of the central limit theorem below is even in total variation.
	By Bernsteins Theorem, there exists for all $t>0$ a Radon measure $\tilde \mu_t$ on $[0, \infty)$ such that for all $r\geq 0$
	\begin{equation*}
	\tilde v(t, \sqrt{2r}) = \int_{[0,\infty)} \tilde\mu_t(\mathrm du) \, e^{-u r}
	\end{equation*}
	where $v(t, \cdot) \eqqcolon \tilde v(t, |\cdot|)$. With $\mu_t \coloneqq \tilde \mu_t \circ \sqrt{\, \cdot\, }^{\, -1}$ we have for all $t>0$ and $x\in \R^d$
	\begin{equation*}
	v(t, x) = \int_{[0,\infty)} \mu_t(\mathrm du) \, e^{-u^2|x|^2/2}.
	\end{equation*}
	Hence, we can further rewrite for $A\in \mathcal A$
	\begin{align}
	\label{representation of P by Q}
	\mathbf P_{\xi}(A) \nonumber
	=& \frac{1}{F(\xi)}\int_A \mathcal W(\mathrm d \mathbf x) \, \prod_{i=1}^{n} v(t_i-s_i, \mathbf x_{s_i, t_i}) \nonumber \\
	&= \frac{1}{F(\xi)}\int_{[0, \infty)^{n}} \bigotimes_{i=1}^n\mu_{t_i-s_i}(\mathrm du_i)\int_A  \mathcal W(\mathrm d\mathbf x) \, e^{ -\frac{1}{2} \sum_{i=1}^{n} u_i^2 |\mathbf x_{s_i, t_i}|^2}.
	\end{align}
	We normalize the inner expression by marking a point process with distribution $\widehat \Gamma_{\alpha, T}$ accordingly.
	Given a locally compact Polish space $E$ we can identify a probability measure on $(\mathbf N_f(E), \mathcal N_f(E))$ with a symmetric probability measure on $E^\cup \coloneqq \bigcup_{n=0}^\infty E^n$. Let $E_1$ and $E_2$ be two locally compact Polish spaces and $\kappa:E_1^\cup \times \mathcal B(E_2^\cup) \to [0, 1]$ be a probability kernel satisfying
	\begin{enumerate}
		\item $\kappa(x, E_2^n) = 1$ for all $x\in E_1^n, n\in \N_0$
		\item $\kappa(\sigma x, \sigma A) = \kappa (x, A)$ for all $x\in E_1^n, A\in \mathcal B(E_2^n)$, $\sigma \in S_n$ and $n\in \N$
	\end{enumerate}
	where $\sigma x \coloneqq (x_{\sigma(1)}, \hdots, x_{\sigma(n)})$ for $x\in E_i^n$, a permutation $\sigma\in S_n$ and $i\in \{1, 2\}$. Given a probability measure $\mathcal P$ on $\mathbf N_f(E_1)$ we can define the marked distribution $\mathcal P \otimes \kappa$ in the following way: Draw a sample $\sum_{i=1}^n \delta_{x_i}$ according to $\mathcal P$, draw marks $(y_1,\hdots, y_n)$ according to $\kappa((x_1, \hdots,x_n), \cdot)$ and obtain $\sum_{i=1}^n \delta_{(x_i, y_i)}$. That is, under the identification mentioned above,
	\begin{equation*}
	 \big(\mathcal P \otimes \kappa\big)(\mathrm d x \, \mathrm d  y) = \mathcal P(\mathrm dx) \kappa(x, \mathrm dy).
	\end{equation*}
	We apply this to our case and define for $\zeta = \sum_{i=1}^{n}\delta_{(s_i, t_i, u_i)} \in \mathbf N_f(\triangle \times [0, \infty))$ the centered Gaussian measure $\mathbf Q_\zeta$ on $(C(\R, \R^d), \mathcal A)$ by
	\begin{equation*}
		\mathbf Q_{\zeta}(\mathrm d \mathbf x)= \frac{1}{\phi(\zeta)} \exp\Big(-\frac{1}{2} \sum_{i=1}^{n} u_i^2 |\mathbf x_{s_i, t_i}|^2 \Big)  \, \mathcal W(\mathrm d \mathbf x)
	\end{equation*}
	where $\phi(\zeta)$ is a normalization constant and $\kappa: \triangle^\cup \times \mathcal B([0, \infty)^\cup) \to [0, 1]$ by
	\begin{equation*}
	\kappa(\xi, \mathrm du) \coloneqq \frac{\phi(\xi, u)}{ F(\xi)} \, \bigotimes_{i=1}^n\mu_{t_i-s_i}(\mathrm du_i)
	\end{equation*}
	with $\phi(\xi, u) \coloneqq 0$ for $\xi\in \triangle^n$, $u\in [0, \infty)^m$ with $n \neq m$. By Equation \eqref{representation of P by Q} we then obtain
	\begin{equation*}
	\mathbb P_{\alpha, T}(\cdot) = \int_{\mathbf N_f(\triangle \times [0, \infty))} \big(\widehat{\Gamma}_{\alpha, T}\otimes \kappa\big)(\mathrm d \zeta) \, \mathbf Q_{\zeta}(\cdot).
	\end{equation*}
	Notice that local convergence of the measures $\widehat \Gamma_{\alpha, T}$ immediately implies local convergence of the measures $\widehat \Gamma_{\alpha, T} \otimes \kappa$.
	Provided that \eqref{Our condition} is satisfied, one can obtain $\widehat \Gamma_{\alpha, T}\otimes \kappa$ by marking the tilted clusters independently of each other. That is, starting in $-T$ we alternate independently drawn $\operatorname{Exp}(\psi(\alpha))$ distributed dormant periods with $\widehat{\Xi}_{\alpha}\otimes \kappa $ distributed marked clusters. Conditionally on the event that the system is dormant at $T$, the process of marked customers arriving between $-T$ and $T$ has distribution $\widehat{\Gamma}_{\alpha,T} \otimes \kappa$. In particular, the explicit form of the infinite volume measure is the same if we replace the measures $\mathbf P_\xi$ with the measures $\mathbf Q_\zeta$ and $\widehat \Xi_{\alpha}$ with $\widehat \Xi_{\alpha} \otimes \kappa$. \\
	For the special case of the Polaron, i.e. $w(t, x) \coloneqq e^{-t}/|x|$ and $w(t, 0) \coloneqq \infty$ for $t \geq 0$ and $x\in \R^3\setminus \{0\}$, one chooses $g(t) = e^{-t}$ for $t\geq 0$. Since
	\begin{equation*}
	\frac{1}{|x|} = \sqrt{\frac{2}{\pi}} \int_0^\infty \mathrm du \, e^{-u^2|x|^2/2}
	\end{equation*}
	for all $x\in \R^3\setminus \{0\}$, one obtains the representation derived in \cite{MV19} of the Fröhlich polaron as a mixture of Gaussian measures.
\end{remark}

\section{Proof of the functional central limit theorem}
\label{Section: Functional central limit theorem}
We equip $C(\R, \R^d)$ with the topology of locally uniform convergence. A function $f:\R^n \to \R\cup \{\infty\}$ is called quasiconcave (quasiconvex) if for all $a\in \R \cup \{\infty\}$ the superlevel set $f^{-1}([a, \infty])$ is convex (the sublevel set $f^{-1}((-\infty, a])$ is convex). Provided that $w(t, \cdot)$ is quasiconcave for all $t>0$ and that Conditions (A1)-(A3) are met,
we are going to show that the distribution of $X^n$ under $\mathbb P_{\alpha, \infty}$ converges weakly to the distribution of a centered Gaussian process with stationary and independent increments. In case that $w(t, \cdot)$ is rotationally symmetric for all $t>0$, this Gaussian process is a rescaled Brownian motion.\\
\\
We start by showing tightness of $\{\mathbb P_{\alpha, \infty} \circ (X^n)^{-1}: \, n\in \N\}$. For $n \in \N$ we define the modulus of continuity by
\begin{equation*}
\omega_n: C(\R, \R^d) \times (0, \infty) \to [0, \infty), \quad \omega_n(\mathbf x, \delta) \coloneqq \sup_{s, t \in [-n, n], |s-t| < \delta} |\mathbf x_{s, t}|.
\end{equation*}
The Gaussian correlation inequality \cite{article} states that for all convex sets $A_1, A_2 \subseteq \R^n$ that are symmetric about the origin and any centered Gaussian measure $\mu$ on $\R^n$
\begin{equation*}
\mu(A_1 \cap A_2) \geq \mu(A_1) \mu(A_2) .
\end{equation*}
It is well known that this implies $\mathbb E_{\mu}[fg] \geq \mathbb E_{\mu}[f] \mathbb E_{\mu}[g]$ for all non-negative, symmetric, quasiconcave functions $f, g$ and any centered Gaussian measure $\mu$ on $\R^n$. We generalize this and obtain the following proposition:

\begin{prop}
	\label{Gaussian correlation inequality}
	Let $X$ be a $n$-dimensional centered Gaussian vector.
	\begin{enumerate}
		\item If $f_1, \hdots, f_k: \R^n \to [0, \infty]$ are symmetric (with respect to point reflections in the origin) and quasiconcave and $f_{k+1}:\R^n \to [0, \infty]$ is symmetric and quasiconvex then 
		\begin{equation*}
		\mathbb E\Big[\prod_{i=1}^{k+1} f_i(X) \Big] \leq \mathbb E\Big[\prod_{i=1}^k f_i(X)\Big]  \cdot \mathbb E[f_{k+1}(X)].
		\end{equation*}
		\item If $f_1, \hdots, f_m:\R^n \to [0, \infty]$ are symmetric and quasiconcave then
		\begin{equation*}
		\mathbb E\Big[\prod_{i=1}^m f_i(X)\Big] \geq \prod_{i=1}^{k} \mathbb E\Big[\prod_{j\in J_i} f_j(X)\Big].
		\end{equation*}
		for any partition $\{1, \hdots, m\} = J_1 \dot \cup \hdots \dot \cup J_k$.
	\end{enumerate}
\end{prop}

\begin{proof}
	We only show the first statement, the proof of the second statement can be conducted similarly. 
	As a direct consequence of the Gaussian correlation inequality, one obtains for $A_1 \subseteq \R^n$ symmetric and convex and $A_2 \subseteq \R^n$ symmetric such that $A_2^c$ is convex
	\begin{equation*}
	\mathbb P(X\in A_1, X \in A_2) \leq \mathbb P(X\in A_1) \mathbb P(X \in A_2) .
	\end{equation*}
	We write
	\begin{equation*}
	f_i(X) = \int_0^\infty \1_{[0, f_i(X)]}(s) \, \mathrm d s, \quad \text{ for }1\leq i \leq k, \quad f_{k+1}(X) = \int_0^\infty \1_{[0, f_{k+1}(X))}(s) \, \mathrm d s.
	\end{equation*}
	Now, for all $s_1, \hdots, s_{k+1}>0$ the sets $\bigcap_{i = 1}^k f_{i}^{-1}([s_i, \infty])$ and $(f_{k+1}^{-1}((s_{k+1}, \infty]))^c$ are symmetric and convex. We hence get with the Gaussian correlation inequality
	\begin{align*}
	&\mathbb E\Big[\prod_{i=1}^{k+1} f_i(X)\Big]\\
	&= \int_0^\infty \hdots \int_0^\infty \mathbb P(f_1(X)\geq s_1, \hdots, f_{k+1}(X) > s_{k+1}) \, \mathrm ds_1 \hdots \mathrm ds_{k+1}  \\
	& \leq \int_0^\infty \hdots \int_0^\infty \mathbb P(f_1(X)\geq s_1, \hdots, f_{k}(X) \geq s_{k}) \mathbb P(f_{k+1}(X) > s_{k+1})  \, \mathrm ds_1 \hdots \mathrm ds_{k+1}  \\
	&= \mathbb E\Big[\prod_{i=1}^{k} f_i(X)\Big] \cdot \mathbb E[f_{k+1}(X)].	\qedhere
	\end{align*}
\end{proof}

\begin{cor}
	\label{Covariance estimate} Assume that (A1) holds and that $w(t, \cdot)$ is quasiconcave for all $t>0$.
	Then for all $\xi\in \mathbf N_f(\triangle)$ with $F(\xi)<\infty$ and all $a<b$
	\begin{equation*}
	\mathbf E_{\xi}\big[|X_{a, b}|^2\big] \leq \mathbb E_{\mathcal W}\big[|X_{a, b}|^2\big] = d(b-a). 
	\end{equation*}
\end{cor}	
\begin{proof}
	For $\xi = \sum_{i=1}^k \delta_{(s_i, t_i)} \in \mathbf N_f(\triangle)$ with $F(\xi)<\infty$ we apply Proposition \ref{Gaussian correlation inequality} to the quasiconvex function $f_{k+1}:\mathbb R^{dk+d} \to [0, \infty)$, $f_{k+1}(x_1, \hdots, x_{k+1}) \coloneqq |x_{k+1}|^2$, the quasiconcave functions $f_1, \hdots, f_k:\mathbb R^{dk+d} \to (0, \infty]$ defined by
	\begin{equation*}
	f_i(x_1, \hdots, x_{k+1}) \coloneqq v(t_i-s_i, x_i),
	\end{equation*}
	for $1 \leq i \leq k$ and the Gaussian vector $(X_{s_1, t_1}, \hdots, X_{s_k, t_k}, X_{a, b})$.	
\end{proof}

\begin{remark}
	 Assume that (A1) and (A2) hold. If $w(t, \cdot)$ is quasiconcave for all $t>0$ then Proposition \ref{Gaussian correlation inequality} implies that for $\xi_1, \xi_2 \in \mathbf N_f(\triangle)$
	\begin{equation*}
		F(\xi_1 + \xi_2) \geq F(\xi_1) F(\xi_2)
	\end{equation*}
	i.e. in that sense $\log(F)$ is superadditive. In particular, if \eqref{Our condition}
	holds, then the interaction energy between left and right half axis satisfies
	\begin{equation}
	\label{Equation: Finite interaction energy}
		\mathbb E_\mathcal W\bigg[\exp\bigg( \alpha \int_{-\infty}^0\int_{0}^\infty w(t-s, X_{s, t}) \, \mathrm dt \mathrm ds \bigg)\bigg] < \infty.
	\end{equation}
	In order to see this, we write the Poisson point process $\eta \sim \Gamma_{\alpha, T}$ as the sum $\eta = \eta_1 + \eta_2 + \eta_3$ where $\eta_1, \eta_2, \eta_3$ are independent Poisson point processes with intensity measures 
	\begin{align*}
		  \mu^1_{\alpha, T}(\mathrm ds\mathrm d t) &\coloneqq \alpha \cdot g(t-s) \mathds 1_{\{-T < s < t < 0 \}}\mathrm d s \mathrm dt \\
		\mu^2_{\alpha, T}(\mathrm ds\mathrm d t) &\coloneqq \alpha \cdot g(t-s) \mathds 1_{\{0 < s < t < T\}}\mathrm d s \mathrm dt \\
		 \mu^3_{\alpha, T}(\mathrm ds\mathrm d t) &\coloneqq \alpha \cdot g(t-s) \mathds 1_{\{-T < s < 0 < t < T\}}\mathrm d s \mathrm dt
	\end{align*}
	respectively. Then
	\begin{align*}
		\mathbf Z_{\alpha, 2T} &= \mathbb E[F(\eta_1 + \eta_2 + \eta_3)] \\		
		 &\geq \mathbb E[F(\eta_1)] \mathbb E[F(\eta_2)] \mathbb E[F(\eta_3)] = \mathbf Z_{\alpha, T}^2 \mathbb E[F(\eta_3)].
	\end{align*}
	By the same calculations as in Section \ref{Section: Outline} we have
	\begin{equation*}
		\mathbb E[F(\eta_3)] = e^{-b_{\alpha, T}} \mathbb E_\mathcal W\bigg[\exp\bigg( \alpha \int_{-T}^0\int_{0}^T w(t-s, X_{s, t}) \, \mathrm dt \mathrm ds \bigg) \bigg]
	\end{equation*}
	with $b_{\alpha, T} = \mu_{\alpha, T}^3(\triangle)$ and the statement follows by the monotone convergence theorem and boundedness of $T \mapsto b_{\alpha, T}$. Notice that
	\begin{equation*}
		\frac{Z_{\alpha, T}^2}{Z_{\alpha, 2T}} = \mathbb E_{\alpha, T}\bigg[\exp\bigg(-\alpha \int_{-T}^0\int_{0}^T w(t-s, X_{s, t}) \, \mathrm dt \mathrm ds \bigg) \bigg]
	\end{equation*}
	Therefore, heuristically we would expect \eqref{Our condition} to hold if and only if (assuming an infinite volume measure $\mathbb P_{\alpha, \infty}$ exists)
	\begin{equation*}
		\mathbb E_{\alpha, \infty}\bigg[\exp\bigg(-\alpha \int_{-\infty}^0\int_{0}^\infty w(t-s, X_{s, t}) \, \mathrm dt \mathrm ds \bigg) \bigg] >0
	\end{equation*}
	i.e. $\mathbb P_{\alpha, \infty}\big(\int_{-\infty}^0\int_{0}^\infty w(t-s, X_{s, t}) \, \mathrm dt \mathrm ds < \infty \big) > 0$; it would be interesting to have a rigorous proof (or counterexample) for this connection, as well as some understanding how it relates to Equation \eqref{Equation: Finite interaction energy}.
\end{remark}

\begin{lemma}
	\label{modulus of continuity estimate}
	Assume that (A1) holds and that $w(t, \cdot)$ is quasiconcave for all $t>0$. Then for all $n\in \N$, $\varepsilon, \delta>0$ and $\xi \in \mathbf N_f(\triangle)$ with $F(\xi)<\infty$
	\begin{equation*}
	\mathbf P_{\xi}\big( \{\mathbf x\in C(\R, \R^d): \, \omega_n(\mathbf x, \delta)  > \varepsilon\}\big) \leq \mathcal W\big( \{\mathbf x\in C(\R, \R^d): \, \omega_n(\mathbf x, \delta)  > \varepsilon\}\big).
	\end{equation*}
\end{lemma}
\begin{proof}
	We enumerate
	\begin{equation*}
	[-n, n] \cap \Q  = \{q_1, q_2, q_3, \hdots\}.
	\end{equation*}
	Let $\xi = \sum_{i=1}^m \delta_{(s_i, t_i)} \in \mathbf N_f(\triangle)$ with $F(\xi)<\infty$ and $k\in \N$. We define $f_1, \hdots, f_{m}: \R^{d(k^2+m)} \to (0, \infty]$ by
	\begin{equation*}
	f_i( x_{11},  x_{12},  \hdots,  x_{kk},  y_{1}, \hdots, y_{m}) \coloneqq v(t_i - s_i, y_i)
	\end{equation*}
	for all $1\leq i \leq m$ and  $f_{m+1}: \R^{d(k^2+m)} \to [0, \infty)$ by
	\begin{equation*}
	f_{m+1}(x_{11}, x_{12},  \hdots, x_{kk}, y_{1}, \hdots,  y_{m})
	\coloneqq
	\begin{cases}
	1 \text{ if } &\exists i,j \in \{1, \hdots, k\}: \,  | x_{ij}|>\varepsilon \\
	&\text{ and }|q_i - q_j|< \delta \\
	0 \text{ else}
	\end{cases}.
	\end{equation*}
	Then $f_{m+1}$ is symmetric and quasiconvex and $f_1, \hdots, f_m$ are symmetric and quasiconcave. If we define
	\begin{equation*}
	Y \coloneqq (X_{q_1, q_1}, \hdots, X_{q_1, q_k}, X_{q_2, q_1}, \hdots,  X_{q_k, q_k}, X_{s_1, t_1}, \hdots, X_{s_m, t_m})
	\end{equation*}
	we hence get with the Gaussian correlation inequality
	\begin{align*}
	&\mathbf P_{\xi}\big( \big\{\mathbf x\in C(\R, \R^d): \, (\exists i, j\in \{1, \hdots, k\}: |q_i - q_j| < \delta \text{ and }|\mathbf x_{q_i, q_j}| > \varepsilon  )\big\}\big) \\
	&= \frac{1}{F(\xi)}\mathbb E_\mathcal W\Big[\prod_{i=1}^{m+1} f_i(Y)\Big] \\
	&\leq \mathbb E_\mathcal W[f_{m+1}(Y)]\\
	&= \mathcal W\big( \{\mathbf x\in C(\R, \R^d): \, (\exists i, j\in \{1, \hdots, k\}: |q_i - q_j| < \delta \text{ and }|\mathbf x_{q_i, q_j}| > \varepsilon )\big\}\big).
	\end{align*}
	Writing 
	\begin{align*}
	&\{\mathbf x\in C(\R, \R^d): \, \omega_n(\mathbf x, \delta)  > \varepsilon\} \\
	&= \bigcup_{k=1}^\infty \big\{\mathbf x\in C(\R, \R^d): \, (\exists i, j\in\{1, \hdots, k\}: |q_i - q_j| < \delta \text{ and }|\mathbf x_{q_i, q_j}| > \varepsilon  )\big\}
	\end{align*}
	and using continuity of measures from below yields the claim.		
\end{proof}
\begin{lemma}
	Assume that $w(t, \cdot)$ is quasiconcave for all $t>0$ and that Conditions (A1)-(A3) are met. 
	Then
	the family of measures $\big\{\mathbb P_{\alpha, \infty} \circ (X^n)^{-1}: \, n\in \N\big\}$ is tight.
\end{lemma}
\begin{proof}
	As $\mathbb P_{\alpha, \infty}(X^n_0 = 0) = 1$ for all $n\in \N$ the tightness is equivalent to
	\begin{equation*}
	\forall m\in \N\, \forall \eta>0\, \forall \varepsilon>0 \,\exists \delta>0\, \forall n\in \N: \, \mathbb P_{\alpha, \infty}(\omega_m(X^n, \delta)> \varepsilon) < \eta.
	\end{equation*}
	By Lemma \ref{modulus of continuity estimate}, we have for $n, m\in \N$ and $\varepsilon, \delta >0$
	\begin{align*}
	\mathbb P_{\alpha, \infty}(\omega_m(X^n, \delta) > \varepsilon) &= \mathbb P_{\alpha, \infty}(\omega_{mn}(X, \delta n) >  \sqrt{n}\varepsilon) \\
	&= \mathbb E_{\alpha}\big[ \mathbf P_{R_0^{mn} \hat \eta_s}\big(\omega_{mn}(X, \delta n) >  \sqrt{n}\varepsilon\big) \big] \\
	&\leq \mathbb E_{\alpha}\big[ \mathcal W\big(\omega_{mn}(X, \delta n) >  \sqrt{n}\varepsilon\big) \big] \\
	&= \mathcal W\big(\omega_{mn}(X, \delta n) >  \sqrt{n}\varepsilon\big) \\
	&= \mathcal W(\omega_m(X, \delta) > \varepsilon)
	\end{align*}
	where we used Brownian scaling in the last step. The claim follows by the fact that a single probability measure on a Polish space is tight.
\end{proof}
In order to prove the functional central limit theorem, it is left to show convergence of the finite dimensional distributions. We will use the following lemma. 

\begin{lemma}
	\label{Mixing property}
	Assume Conditions (A1)-(A3).
	Let $s_1 < t_1 < s_2 < t_2 < \hdots < s_k < t_k$. Then, as $n \to \infty$, 
	\begin{equation*}
	\sup_{A_1, \hdots, A_k \in \mathcal B(\R^d)} \Big| \mathbb P_{\alpha, \infty}(X^n_{s_1, t_1} \in A_1, \hdots, X^n_{s_k, t_k} \in A_k) - \prod_{i=1}^k \mathbb P_{\alpha, \infty}(X^n_{s_i, t_i} \in A_i) \Big| \to 0.
	\end{equation*}	
\end{lemma}

\begin{proof}
	Let $((\hat T_0, \hat \gamma), (\hat d_1, \hat \xi_1), (\hat d_2, \hat \xi_2), \hdots)$ be a sequence of independent $(0, \infty) \times \mathbf N_f(\triangle)$ valued random variables such that
	\begin{itemize}
		\item  For all $k\geq 1$ the random variable $(\hat d_k, \hat \xi_k)$ is $\operatorname{Exp}(\psi(\alpha)) \otimes  \widehat \Xi_\alpha$ distributed 
		\item $(\hat T_0, \hat \gamma)$ is as in Proposition \ref{convergence to stationarity} such that $R_0^\infty \hat \eta_s \overset{d}{=} \hat \gamma + \theta_{-\hat T_0} \hat \eta$ (where $\hat \eta$ denotes the process obtained by alternating dormant periods $\hat d_k$ and active clusters $\hat \xi_k$ (starting dormant) and $\hat \eta_s$ denotes a stationary version of $\hat \eta$).
	\end{itemize}
	By stationarity, we may assume w.l.o.g. that $s_1 = 0$. Let $A_1, \hdots, A_k \in \mathcal B(\R^d)$. For $n\in \N$, we define $f_n, g_n: \mathbf N_f(\triangle) \to [0, 1]$ by
	\begin{align*}
		f_n(\xi) &\coloneqq \mathbf P_{\xi}(X^n_{0, t_1} \in A_1, \hdots, X_{s_{k-1}, t_{k-1}}^n \in A_{k-1}) \\
		g_n(\xi) &\coloneqq \mathbf P_{\xi}(X_{0, t_{k}-s_k}^n \in A_{k}).
	\end{align*}
	For $j\geq 1$ let (as usual) $\hat T_j$ denote the sum of the $j$-th dormant and active period. 
	Let $(\hat S_j^s)_{j\geq 0} \coloneqq (\hat T_0, \hat T_0 + \hat T_1, \hdots )$ denote the embedded delayed renewal process of $R_0^\infty \hat \eta_s$ and, for $n\in \N$, 
	\begin{equation*}
		Y_n \coloneqq \inf\{\hat S^s_j - nt_{k-1}: \,j\geq 0 \text{ such that } \hat S^s_j >  nt_{k-1}\}
	\end{equation*}
	denote the forward recurrence time at time $nt_{k-1}$. Let $\varepsilon>0$. By Proposition \ref{convergence to stationarity} and the convergence of the distribution of $Y_n$ as $n\to \infty$ (it does not matter that the renewal process is delayed by $\hat T_0$, see e.g. \cite[p. 155]{asmussen2003applied}), there exists a $N\in \N$ and a $T>0$ such that for all $n\geq N$ and $t \geq T$
	\begin{align*}
		&\hat{\mathbb P}_\alpha\big(Y_n > T\big)< \varepsilon,\quad  n(s_{k} - t_{k-1}) > 2T, \,\\
		&\big\|\hat{ \mathbb P}_{\alpha} \circ (R_0^\infty \hat \eta_s)^{-1}  -
		\hat{\mathbb P}_{\alpha} \circ (R_0^\infty \theta_t \hat \eta)^{-1} \big\| < \varepsilon.
	\end{align*}
	As $n(s_{k} - t_{k-1}) -Y_n(\omega) > T$ for $\omega \in \{Y_n \leq T\}$ and $n\geq N$ and since
	\begin{equation*}
	\mathbb P_{\alpha, \infty}(X^n_{s_k, t_k} \in A_k) = \hat{\mathbb E}_\alpha\big[g_n(R_0^{n(t_k-s_k)} \hat \eta_s)\big]
	\end{equation*}
	conditioning with respect to the process up to the first renewal after $nt_{k-1}$ yields
	\begin{align*}
		&\Big| \mathbb P_{\alpha, \infty}(X^n_{0, t_1} \in A_1, \hdots, X^n_{s_k, t_k} \in A_k)\\
		&- \mathbb P_{\alpha, \infty}(X^n_{0, t_1} \in A_1, \hdots, X^n_{s_{k-1}, t_{k-1}} \in A_{k-1})\mathbb P_{\alpha, \infty}(X^n_{s_k, t_k} \in A_k)  \Big| \\
		&\leq 2 \varepsilon + \int_{\hat \Omega} \hat{\mathbb P}_\alpha(\mathrm d \omega)\, \1_{\{Y_n \leq T\}}(\omega) f_n(R_0^{nt_{k-1}} \hat \eta_s(\omega ))\\
		& \quad \quad \quad \quad \quad \quad \quad \cdot \Big|\hat{\mathbb E}_\alpha\big[g_n\big(R_{0}^{n(t_k -s_k)} \theta_{n(s_k -t_{k-1}) -Y_n(\omega)}  \, \hat \eta \big)\big] - \mathbb P_{\alpha, \infty}(X^n_{s_k, t_k} \in A_k) \Big| \\
		&\leq 4\varepsilon
	\end{align*}
	for all $n\geq N$. The claim follows by inductively applying this argument.
\end{proof}

For a proof of the following theorem by Rényi \cite{Rnyi1963} (for $d=1$) that can directly be generalized to higher dimensions we refer the reader to \cite[p. 346 f.]{Gut2013}.
\begin{thrm}[Anscombe-Rényi]
	\label{Reyni-Theorem}
	Let $(X_n)_n$ be an iid sequence of centered random vectors with $\mathbb E[|X_1|^2] < \infty$ and $(N_t)_t$ be a family of $\mathbb N$-valued random variables such that $N_t/t$ converges in probability to a constant $\theta>0$ as $t\to \infty$. Then
	\begin{equation*}
	\frac{1}{\sqrt{t}} \sum_{k=1}^{N_t} X_k \overset{d}{\rightarrow} \mathcal N(0, \theta \Sigma)
	\end{equation*}
	as $t\to \infty$ where $\Sigma \coloneqq \mathbb E[X_1 \cdot X_1^T]$.	
\end{thrm}

\begin{lemma}
	\label{fidi convergence}
	Assume that $w(t, \cdot)$ is quasiconcave for all $t>0$ and that (A1)-(A3) hold. 
	Let $s_1 < t_1 < \hdots < s_k < t_k$. Then
	\begin{equation*}
	\mathbb P_{\alpha, \infty}\circ (X^{n}_{s_1, t_1}, \hdots, X^n_{s_k, t_k})^{-1} \Rightarrow \bigotimes_{i=1}^k \mathcal N\big(0, (t_i - s_i) \Sigma\big)
	\end{equation*}
	as $n\to \infty$, where 
	\begin{equation*}
	\Sigma \coloneqq  \frac{\hat{\mathbb E}_\alpha[\Sigma(\hat d_1, \hat \xi_1)]}{\hat{\mathbb E}_\alpha[\widehat T_1]} \quad \text{with } \Sigma(\hat d_1, \hat \xi_1)  \coloneqq \hat d_1 I_d+ \mathbf E_{\hat \xi_1}\big[ X_{0, \hat a_1} \cdot X_{0, \hat a_1} ^T \big],
	\end{equation*}
	and $(\hat d_1, \hat \xi_1) \sim \operatorname{Exp}(\psi(\alpha)) \otimes \widehat \Xi_{\alpha}$, $\hat a_1$ is the length of the cluster $\hat \xi_1$ and $\widehat T_1 = \hat a_1 + \hat d_1$.
\end{lemma}

\begin{proof}
	By Lemma \ref{Mixing property}, it is sufficient to show that for all $a<b$
	\begin{equation*}
	\mathbb P_{\alpha, \infty} \circ (X^n_{a, b})^{-1} \Rightarrow \mathcal N\big(0, (b - a)\Sigma\big)
	\end{equation*}
	as $n\to\infty$.  By stationarity, we may assume w.l.o.g. that $a=0$. In order to reduce notation we will additionally assume that $b=1$.\\
	Let
	$\big( (\hat T_0, \hat \gamma, \hat Y^0), (\hat d_1, \hat \xi_1, \hat Y^1), (\hat d_2, \hat \xi_2, \hat Y^2), \hdots \big)$ be a sequence of independent $(0, \infty) \times \mathbf N_f(\triangle) \times C(\R, \R^d)$ valued random variables such that
	\begin{itemize}
		\item  For all $k\geq 1$ the random variable $(\hat d_k, \hat \xi_k)$ is $\operatorname{Exp}(\psi(\alpha)) \otimes  \widehat \Xi_\alpha$ distributed and $\hat Y^k$ has conditionally on $(\hat d_k, \hat \xi_k)$ distribution $\mathbf P_{\theta_{-\hat d_k} \hat \xi_k}$
		\item $(\hat T_0, \hat \gamma)$ is as in Proposition \ref{convergence to stationarity} such that $R_0^\infty \hat \eta_s \overset{d}{=} \hat \gamma + \theta_{-\hat T_0} \hat \eta$ (where $\hat \eta$ denotes as usual the process obtained by alternating dormant periods $\hat d_k$ and active clusters $\hat \xi_k$ (starting dormant) and $\hat \eta_s$ denotes a stationary version of $\hat \eta$)
		\item $\hat Y^0$ has conditionally on $(\hat T_0, \hat \gamma)$ distribution $\mathbf P_{\hat \gamma}$.
	\end{itemize}
	In particular
	\begin{equation*}
	\hat {\mathbb E}_\alpha\Big[\hat Y_{0, \hat T_1}^1 \cdot (\hat Y_{0, \hat T_1}^1)^T\Big] = \hat{\mathbb E}_\alpha[\Sigma(\hat d_1, \hat \xi_1)]
	\end{equation*}
	where $\hat T_k$ denotes for $k\in \N$ as usual the sum of the $k$-th active and dormant period. Notice that the existence and finiteness of the second moments follows from 
	\begin{equation*}
	\hat{\mathbb E}_\alpha\Big[\big|\hat Y^1_{0, \hat T_1}\big|^2\Big] \leq d \cdot \hat{\mathbb E}_\alpha[\hat T_1] < \infty
	\end{equation*}
	(remember Corollary \ref{Covariance estimate}). Let $(\hat S_n)_{n\geq 0} = (0, \hat T_1, \hat T_1 + \hat T_2, \hdots)$ denote the embedded renewal process of $\hat \eta$.
	For $t\geq 0$ let
	\begin{equation*}
	\hat X_t \coloneqq
	\sum_{k=1}^{n} \hat Y^{k}_{0, \hat T_{k}} + \hat Y^{n+1}_{0, t-\hat S_{n}} \quad \text{ where $n\in \N_0$ is such that } t\in [\hat S_{n}, \hat S_{n+1}).
	\end{equation*}
	For $t \geq 0$ let $\hat A_t$ denote the forward recurrence time of the embedded renewal process of $\hat \eta$.
	By conditioning with respect to $(\hat T_0 , \hat \gamma, \hat Y^0)$ we get with $V_n \coloneqq n -\hat T_0$ for $f:\R^d \to \R$ continuous and bounded
	\begin{align}
	\label{equations weak convegrence of distribution of increment}
	&\mathbb E_{\alpha, \infty}\big(f(X^n_1)\big) 
	=\hat{\mathbb E}_\alpha\Big[ \mathbf E_{R_0^{n} \hat \eta_s}\big[f(X^n_1) \big]\Big] \nonumber \\
	&= \int_{ \hat \Omega} \hat{\mathbb P}_\alpha(\mathrm d \omega_1) \int_{ \hat \Omega} \hat{\mathbb P}_\alpha(\mathrm d \omega_2) \,
	f\bigg(
	\tfrac{1}{\sqrt{n}} \hat Y^0_{0, \hat T_0(\omega_1) \wedge n}(\omega_1) \nonumber \\
	&\quad  +  \1_{ \{\hat T_0(\omega_1) <n\} } \bigg[ \tfrac{1}{\sqrt{n}} \sum_{k=1}^{\hat \tau_{V_n(\omega_1)}(\omega_2) } \hat Y^k_{0, \hat T_k(\omega_2)}(\omega_2) -  \tfrac{1}{\sqrt{n}}\hat X_{V_n(\omega_1), V_n(\omega_1) + \hat A_{V_n(\omega_1)}(\omega_2)}(\omega_2) \bigg]  \bigg).
	\end{align}	
	Now, fix $\omega_1 \in \hat \Omega$ and set $r \coloneqq \hat T_0(\omega_1)$. We have by Theorem \ref{Elementary Renewal Theorem}
	\begin{equation*}
	\frac{\hat \tau_{n -r}}{n} \to \frac{1}{\hat{ \mathbb E}[\widehat T_1]}
	\end{equation*}
	almost surely as $n\to \infty$. By Theorem \ref{Reyni-Theorem}
	\begin{equation*}
	\frac{1}{\sqrt{n}} \sum_{k=1}^{\hat \tau_{n -r}} \hat Y_{0, \hat T_k}^k \overset{d}{\rightarrow}\mathcal N(0, \Sigma)
	\end{equation*}
	as $n\to \infty$. For $\varepsilon, \delta>0$ let $N \in \N$ be such that for all $n\geq N$
	\begin{equation*}
		\|\hat{\mathbb P}_{\alpha}\circ (\hat X_{n-r, n-r + \hat A_{n-r}})^{-1} - \hat{\mathbb P}_{\alpha}\circ (\hat Y^0_{0, \hat T_0})^{-1} \| < \delta 
	\end{equation*}
	and $\hat{\mathbb P}_\alpha(|\hat Y^0_{0, \hat T_0}| > \sqrt{n}\varepsilon) < \delta$.
	Then, for all $n\geq N$
	\begin{equation*}
		\hat{\mathbb P}_\alpha\Big(\tfrac{1}{\sqrt{n}} \big| \hat X_{n- r, n-r + \hat A_{n - r}} \big| > \varepsilon \Big) < 2 \delta.
	\end{equation*}	
	Hence, by Slutsky's theorem, the inner integral in Equation \eqref{equations weak convegrence of distribution of increment} converges for all $\omega_1 \in \hat \Omega$ to $\int_{\R^d} \mathcal N(0, \Sigma)(\mathrm dx) f(x)$. Thus, by dominated convergence,
	\begin{equation*}
	\lim_{n\to \infty} \mathbb E_{\alpha, \infty}\big(f(X^n_1)\big)  = \int_{\R^d} \mathcal N(0, \Sigma)(\mathrm dx) f(x). \qedhere
	\end{equation*}
\end{proof}

\begin{proof}[Proof of Theorem \ref{Theorem: Functional central limit theorem, long version}]
	We briefly give the argument why we could exclude the case $t_i= s_{i+1}$ for some $1\leq i \leq k-1$ in Lemma \ref{fidi convergence}. Let $(n_k)_k$ be a strictly increasing sequence of natural numbers. By tightness, there exists a subsequence $(n_{k_j})_j$ and a measure $\mathcal P$ on $C(\R, \R^d)$ such that
	\begin{equation*}
	\mathbb P_{\alpha, \infty} \circ (X^{n_{k_j}})^{-1} \Rightarrow \mathcal P
	\end{equation*}
	as $j\to \infty$. By the continuous mapping theorem and Lemma \ref{fidi convergence}, for any $s_1 < t_1 < \hdots < s_k < t_k$
	\begin{equation*}
	\mathcal P \circ (X_{s_1, t_1}, \hdots X_{s_k, t_k})^{-1} = \tilde{\mathcal P} \circ (X_{s_1, t_1}, \hdots X_{s_k, t_k})^{-1}
	\end{equation*}
	where $\tilde{\mathcal P}$ denotes the distribution of $\sqrt{\Sigma} X$ under $\mathcal W$. By approximation, we obtain for all $s_1 < t_1 \leq \hdots \leq s_k < t_k$
	\begin{equation*}
	\mathcal P \circ (X_{s_1, t_1}, \hdots X_{s_k, t_k})^{-1} = \tilde{\mathcal P} \circ (X_{s_1, t_1}, \hdots X_{s_k, t_k})^{-1}
	\end{equation*}
	i.e. $\mathcal P = \tilde{\mathcal P}$. Hence, each subsequence of $(\mathbb P_{\alpha, \infty} \circ (X^{n})^{-1})_n$ has a subsequence that converges weakly to $\tilde{\mathcal P}$. This implies that $(\mathbb P_{\alpha, \infty} \circ (X^{n})^{-1})_n$ converges weakly to $\tilde{\mathcal P}$. It is left to show $\Sigma \leq I_d$. By the Gaussian correlation inequality in the form of Proposition \ref{Gaussian correlation inequality}, we have for all $x\in \R^d$
	\begin{equation*}
	\mathbf E_{\hat \xi_1}\big[\langle x,  X_{0, \hat a_1}  \rangle^2\big] \leq  \mathbb E_\mathcal W[ \langle x,  X_{0, \hat a_1}  \rangle^2] = \hat a_1 |x|^2
	\end{equation*}
	and hence
	\begin{align*}
	\big \langle x, \hat{\mathbb E}_\alpha[ \Sigma(\hat d_1, \hat \xi_1)] x \big \rangle 
	&= \hat{\mathbb E}_{\alpha}[\hat d_1] |x|^2  + \hat{\mathbb E}_\alpha \bigg[ \sum_{i, j=1}^d x_ix_j \mathbf E_{\hat \xi_1}\Big[X_{0, \hat a_1}^i X_{0, \hat a_1}^j\Big]  \bigg] \\
	&=  \hat{\mathbb E}_{\alpha}[\hat d_1]  |x|^2 + \hat{\mathbb E}_{\alpha} \Big[ \mathbf E_{\hat \xi_1}\big[ \langle x,  X_{0, \hat a_1}  \rangle^2 \big]\Big]\\
	&\leq \mathbb E[\hat T_1]|x|^2. \qedhere
	\end{align*}
\end{proof}

\begin{remark}
	If $w(t, \cdot)$ is additionally rotationally symmetric for all $t>0$ then $\Sigma$ is a multiple of the unit matrix, i.e. the limiting distribution is a rescaled Wiener measure.
\end{remark}

\section{Proof of Proposition  \ref{prop:sufficientConditionsForGC} c)}
\label{Section: Growth condition for Polaron-type models}

In this section we assume that $w(t,\cdot)$ is rotationally symmetric and positive (allowing the value $+\infty$) for all $t\geq0$, and that the function $\tilde w$ defined by $w(\cdot,\cdot) = \tilde w (\cdot, |\cdot|)$ is decreasing in the second variable. Remember that we defined
$\tilde w_\beta = \sup_{t \geq 0} \e{\beta t} \tilde w(t,\cdot)$ for $\beta>0$. Let $\beta>0$ and $p>1$ be such that the integrability condition in Proposition \ref{prop:sufficientConditionsForGC} c) holds. We choose 
$g(t)= \beta \e{-\beta t}$ in the decomposition
\eqref{eq: w decomposition} and obtain 
\[
v(t,x) = \frac{1}{\beta} \e{\beta t} w(t,x) \leq \frac{1}{\beta} \tilde w_\beta(|x|)
\]
for all $t\geq 0$ and $x\in \R^d$.
Notice that $\tilde w_\beta$ is decreasing and (as $\tilde w_\beta$ additionally satisfies the integrability condition) that $\tilde w_\beta(r)<\infty$ for all $r\in (0, \infty)$. Without loss of generality we can assume that $v(t,x) > 1$ for all $t\geq 0$ and $x\in \R^d$. The reason is that the function $\bar w$ defined by 
$\bar w(t,x) = w(t,x) + \beta \e{-\beta t}$ satisfies the assumptions of Proposition 
\ref{prop:sufficientConditionsForGC} c) if $w$ does, and leads to $v > 1$. As mentioned at the end of Section \ref{Section: Outline}, \eqref{Our condition} holds for $\bar w$ if and only if it holds for $w$.

Let $(\sigma_n)_{n\geq 0}$ be the iid sequence of $\operatorname{Exp}(\alpha)$ distributed interarrival times and let $(\tau_n)_{n \geq 1}$ be the iid sequence of $\operatorname{Exp}(\beta)$ distributed service times (which is independent of $(\sigma_n)_n$) of our queue. For $n\in \N$ let
\begin{equation*}
	s_n \coloneqq \sum_{k=0}^{n-1} \sigma_k, \quad t_n \coloneqq s_n+ \tau_n.
\end{equation*}
That is, $s_n$ is the arrival and $t_n$ the departure of the $n$-th customer and $\sigma_n$ is the time that passes between the arrival of the $n$-th and the $n+1$-th customer. Then the first cluster $\xi_1$ is given by
\begin{equation*}
	\xi_1 = \sum_{i=1}^N \delta_{(s_i-s_1, t_i-s_1)}
\end{equation*}
where
\begin{equation*}
N = \inf \big\{n\in \N: \tau_k < s_{n+1} - s_k \, \text{ for all } 1 \leq k \leq n\big\}
\end{equation*}
is the number of customers in the first cluster. Notice that $N$ is a stopping time with respect to the filtration generated by $((\sigma_n, \tau_n))_{n \geq 1}$. 
We define the function 
\[
h:(0, \infty) \times (0, \infty) \to (0, \infty], \quad h(t_1, t_2) \coloneqq \mathbb E_\mathcal W[v(t_1, X_{t_2})].
\]
\begin{lemma}
\label{Estimate on F}
We have
	\begin{equation*}
	\prod_{i=1}^{N} h(\tau_i, \tau_i) \leq F(\xi_1) \leq \prod_{i=1}^{N} h(\tau_i, \sigma_i \wedge \tau_i).
	\end{equation*}

\end{lemma}

\begin{proof}
	The lower estimate follows immediately from the Gaussian correlation inequality in the form of Proposition \ref{Gaussian correlation inequality}.
	For the upper bound we distinguish between two cases: If $t_1 \leq s_2$ we immediately get by independence of increments
	\begin{equation*}
	F(\xi_1) = \mathbb E_{\mathcal W}\Big[\prod_{i=1}^{N}v(\tau_i,  X_{s_i, t_i})\Big] = h(\tau_1, \tau_1) \mathbb E_{\mathcal W}\Big[\prod_{i=2}^{N}v(\tau_i, X_{s_i, t_i})\Big].
	\end{equation*}
	In case that $t_1>s_2$ we have by independence of $X_{s_1, s_2}$ and $(X_{s, t})_{s, t\geq s_2}$ 
	\begin{equation*}
	F(\xi_1) =  \int_{C(\R, \R^d)} \mathcal W(\mathrm d\mathbf x) \, \mathbb E_{\mathcal W}\big[v\big(\tau_1 , X_{s_1, s_2} + \mathbf x_{s_2, t_1}\big) \big] \prod_{i=2}^{N}v\big(\tau_i, \mathbf x_{s_i, t_i}\big).
	\end{equation*}
	The assumptions on $w$ imply $\mathbb E_{\mathcal W}\big[v\big(\tau_1, X_{s_1, s_2} + x \big) \big] \leq \mathbb E_{\mathcal W}\big[v\big(\tau_1, X_{s_1, s_2}\big) \big]$ for all $x\in \R^d$ and thus
	\begin{equation*}
	F(\xi_1) \leq  h(\tau_1, \sigma_1)  \mathbb E_{\mathcal W}\Big[\prod_{i=2}^{N}v(\tau_i, X_{s_i, t_i}) \Big].
	\end{equation*}
	Iterating this procedure completes the proof.
\end{proof}

For the Fröhlich polaron path measure we choose $\beta=1$ and obtain by an integration in spherical coordinates
\begin{equation*}
	h(t_1, t_2) = \sqrt{\frac{2}{\pi}} \frac{1}{\sqrt{t_2}}
\end{equation*}
for all $t_1, t_2>0$. Hence, our lower estimate for $F$ becomes the lower estimate given in \cite{MV19}. If one sorts the customers in the first cluster by the time of their departure, starting with the customer that departs last, and modifies our proof accordingly, one obtains an upper estimate for $F$ that is sharper then the estimate given in \cite{MV19}.

\begin{lemma} 
\label{one interval estimate} 
Assume that $w$, $\beta$ and $p$ satisfy the assumptions of Proposition \ref{prop:sufficientConditionsForGC} c). 
Then for all $\alpha > 0$, 
\[
\bbE_\alpha\big[h(\tau_1,\sigma_1 \wedge \tau_1)^p \big] \leq (\alpha+\beta)^{(d+2)/4} c(w,\beta,p),
\]
where the constant $c(w,\beta,p)$ is independent of $\alpha$.
\end{lemma} 

\begin{proof} 
By Jensens inequality,
\[ 
	\mathbb E_\alpha[h(\tau_1, \sigma_1 \wedge \tau_1)^p] \leq \mathbb E_\alpha \Big[\mathbb E_\mathcal W\big[ v(\tau_1, X_{\sigma_1 \wedge \tau_1})^p \big] \Big] 
	\leq \beta^{-p} \cdot\mathbb E_\alpha \Big[\mathbb E_\mathcal W\big[\tilde w_\beta(|X_{\sigma_1 \wedge \tau_1}|)^p \big] \Big].
\]
	For $a, b>0$ we have
	\begin{align*}
	\int_0^\infty \mathrm dt\, t^{-d/2}\exp\big(-a/t - b \cdot t \big) &= \int_0^\infty \mathrm dt\, t^{d-3}\exp\big(-at^2 - b/t^2 \big) \\
	&= \left(b/a\right)^{(d-2)/4} K_{(d-2)/2}(2 \sqrt{ab})
	\end{align*}
	(the second equality can be obtained by substituting $t= (b/a)^{1/4} e^{u}$) where, for $\gamma\in \R$ and $x>0$
	\begin{equation*}
		K_\gamma(x) = \int_0^\infty e^{-x \cosh(u)} \cosh(\gamma u) \, \mathrm du
	\end{equation*}
	denotes the modified Bessel function of second kind. With appropriately chosen constants, $c_d, \tilde c_d>0$ we obtain
	\begin{align}
	&\mathbb E_\alpha \Big[\mathbb E_\mathcal W\Big[ \tilde w_\beta( |X_{\sigma_1 \wedge \tau_1}|)^p \Big] \Big]
	\nonumber  \\
	&= c_d(\alpha + \beta)\int_0^\infty \mathrm dt \, e^{-(\alpha + \beta)t} \int_0^\infty \mathrm dr\,  r^{d-1}  \tilde w_\beta(r)^p e^{-r^2/2t} t^{-d/2} \nonumber \\
	&= \tilde c_d(\alpha + \beta)^{(d+2)/4} \int_0^\infty \mathrm dr \, r^{d-1}  \tilde w_\beta(r)^p  r^{-(d-2)/2} K_{(d-2)/2}(\sqrt{2(\alpha + \beta)} \cdot r)  \nonumber    \\
	&\leq \tilde c_d(\alpha + \beta)^{(d+2)/4} \int_0^\infty \mathrm dr \,  \tilde w_\beta(r)^p \cdot r^{d/2} K_{(d-2)/2}(\sqrt{2 \beta } \cdot r).
	\label{eq: Bessel}
	\end{align}
	Now, as $r \to \infty$, we have \cite[p. 378]{AbraSteg72}
	\begin{align*}
		K_{(d-2)/2}(\sqrt{2 \beta } \cdot r) \in
		\mathcal O\big( r^{-1/2} e^{-\sqrt{2 \beta} \cdot r} \big),
	\end{align*}
	and since $\tilde w_\beta$ is decreasing and thus bounded on $[1, \infty)$, the integral \eqref{eq: Bessel} is finite on $[1,\infty)$.
	For $r\to 0$ we have \cite[p. 375]{AbraSteg72}
	\begin{align*}
		K_{(d-2)/2}(\sqrt{2 \beta } \cdot r) \in 
		\begin{cases}
			\mathcal O\big(-\ln(r)\big) \quad &\text{ for }d=2 \\
			\\ \mathcal O\big(r^{-|d-2|/2}\big) \quad &\text{ for }d\neq 2.
		\end{cases}
	\end{align*}
	The integrability assumptions in  \ref{prop:sufficientConditionsForGC} c) then guarantee that the integral  
	\eqref{eq: Bessel} is finite on $[0,1]$, too, proving the claim. 
\end{proof} 

\begin{lemma}
	\label{estimating ev random product}
	Let $(X_n)_n$ be an iid sequence of positive random variables adapted to some filtration $(\mathcal F_n)_n$. Let $X_{n+1}$ be independent of $\mathcal F_n$ for all $n\in \N$. Let $\tau$ be an a.s. finite stopping time with respect to $(\mathcal F_n)_n$. Assume that $p>1$ is such that $\mathbb E[X^p]< \infty$. Then
	\begin{equation*}
	\mathbb E\Big[ \prod_{i=1}^{\tau} X_i \Big] \leq   \mathbb E\bigg[\Big(\mathbb E\big[X_1^p\big]^{q/p}\Big)^{\tau}\bigg]^{1/q}
	\end{equation*}
	where $q$ denotes the conjugated Hölder index of $p$.
\end{lemma}

\begin{proof}
	For $n\in \N$ define
	\begin{equation*}
	Y_n \coloneqq \prod_{i=1}^{n} \frac{X_i^p}{\mathbb E[X_1^p]}.
	\end{equation*}
	Then $(Y_n)_n$ is a martingale with respect to $(\mathcal F_n)_n$.
	By the optional stopping theorem, we have $\mathbb E[Y_{\tau \wedge n}] = 1$ for all $n\in \N$. By Fatou's Lemma
	\begin{equation*}
	\mathbb E[Y_{\tau}] \leq \liminf_{n \to \infty} \mathbb E[Y_{\tau \wedge n}] = 1.
	\end{equation*}		
	If we denote by $q$ the conjugated Hölder index to $p$ we get with Hölders inequality
	\begin{align*}
	\mathbb E \Big[ \prod_{i=1}^{\tau} X_i \Big] &= \mathbb E \bigg[\mathbb E[X_1^p]^{\tau/p} \prod_{i=1}^{\tau} \frac{X_i}{\mathbb E[X_1^p]^{1/p}} \bigg] \\
	&\leq \mathbb E \Big[\mathbb E[X_1^p]^{\tau q/p} \Big]^{1/q} \mathbb E [Y_\tau]^{1/p} \\
	&\leq \mathbb E \Big[\mathbb E[X_1^p]^{\tau q/p} \Big]^{1/q}. \qedhere
	\end{align*}
\end{proof}

\begin{proof}[Proof of Proposition \ref{prop:sufficientConditionsForGC} c)]
We first show that $Z_{\alpha, T} < \infty$ for all $\alpha, T > 0$ and that $\psi(\alpha)<\infty$ for all $\alpha >0$. 
Notice that, by the series expansion of the exponential function and by a change of the order of integration (in the same manner as in the beginning of Section \ref{Section: Outline}), one obtains Equation \eqref{Equation: partition function and fat partition function} even without assuming that $\mathbb P_{\alpha, T}$ defines a probability measure (i.e. without assuming that $Z_{\alpha, T} < \infty$).
Since $T_1 \geq \sigma_0 + \hdots + \sigma_{N-1} + \tau_N$, Lemma \ref{Estimate on F} implies  
	\begin{equation}
		\label{Equation: Show partition function is finite}
		\mathbb E_\alpha[ e^{-\mu T_1} F(\xi_1)] \leq \mathbb E_\alpha \Big[\prod_{i=1}^{N} e^{-\mu(\sigma_i \wedge \tau_i)} h(\tau_i, \sigma_i \wedge \tau_i) \Big].
	\end{equation}
	for all $\mu \geq 0$.
By Lemma \ref{one interval estimate}, we have $\mathbb E_\alpha[h(\tau_1, \sigma_1 \wedge \tau_1)] < \infty$. 
Hence, there exists a $\mu \geq 0$ such that $\mathbb E_\alpha[e^{-\mu(\sigma_1 \wedge \tau_1)} h(\tau_1, \sigma_1 \wedge \tau_1)] \leq 1$.

As in the proof of Lemma \ref{estimating ev random product}, an application of the optional stopping theorem to the supermartingale
$\big(\prod_{i=1}^n e^{-\mu(\sigma_i \wedge \tau_i)} h(\tau_i, \sigma_i \wedge \tau_i)\big)_{n \in \bbN}$ gives  
$\mathbb E_\alpha[ e^{-\mu T_1} F(\xi_1)] \leq 1$ for this $\mu$. By
Proposition \ref{Proposition: Partition function locally bounded}, there exist $C, \lambda>0$ such that $\mathbf Z_{\alpha, T} \leq C \e{\lambda T}$ for all $T\geq 0$. Since $e^{c_{\alpha, T}} \sim e^{2\alpha T  - \alpha\mathbb E[\tau_1]}$
this yield the existence of $\widetilde C, \tilde \lambda > 0$ such that
	\begin{equation*}
		Z_{\alpha, T} \leq \widetilde C e^{\tilde \lambda T}
	\end{equation*}
	for all $T\geq 0$. 	
	By superadditivity
	\begin{equation*}
		\psi(\alpha) = \lim_{T \to \infty} \frac{\log(Z_{\alpha, T})}{T} = \sup_{T>0}  \frac{\log(Z_{\alpha, T})}{T} \leq \tilde \lambda. 
	\end{equation*} 	
For showing the validity of \eqref{Our condition}, we show that \eqref{goodness of alpha} holds, and as we already 
know that (A1) and (A2) hold, we may then apply Theorem \ref{Theorem: Equvialent conditions for 
goodness in the setting of path measures}. 
 Since we assumed $v>1$, we have $\mathbb E_\alpha[F(\xi_1)]>1$ for all $\alpha >0$, 
 and by Proposition \ref{Remark: stronger version of goodness} it is thus sufficient to show that 
 $\mathbb E_\alpha[F(\xi_1)]< \infty$ for sufficiently small $\alpha$.
	 By Lemma \ref{Estimate on F} it is sufficient to show that
	 \begin{equation*}
	 	\mathbb E_\alpha \Big[\prod_{i=1}^{N} h(\tau_i, \sigma_i \wedge \tau_i) \Big] < \infty
	 \end{equation*}
	 for sufficiently small $\alpha$.
	 Let $r_{\alpha}$ be the radius of convergence of the probability generating function of $N$. By Lemma \ref{estimating ev random product}, it is sufficient to show that
	 \begin{equation}
	 	\label{Equation: h has to be smaller then power of radius of convergence}
	  	\mathbb E_\alpha\big[h(\tau_1, \sigma_1 \wedge \tau_1)^p\big] < r_{\alpha}^{p/q}
	 \end{equation}
	 for sufficiently small $\alpha$. One can convince oneself (e.g. by looking at the known formula of the probability 
	 generating function of $N$ for this particular choice of $g$, see \cite{10.2307/1428137}) that $\lim_{\alpha \to 0}
	 r_{\alpha} = \infty$. On the other hand, Lemma \ref{one interval estimate} shows that the left hand side of 
	 \eqref{Equation: h has to be smaller then power of radius of convergence} remains bounded as $\alpha \to 0$. 
	 This shows the claim.  
\end{proof}

\begin{figure}
	\centering
	\includegraphics[width=0.7\linewidth]{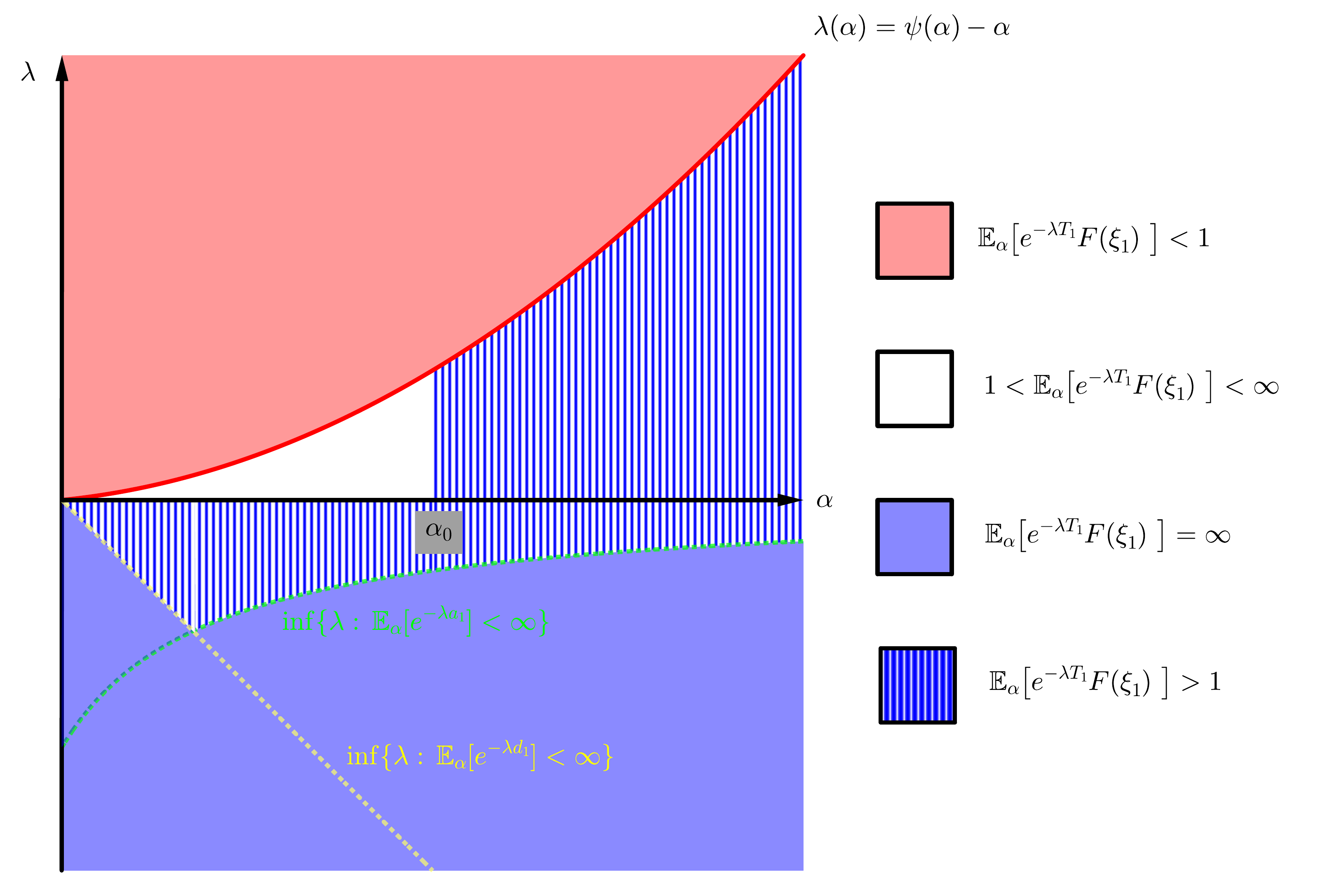}
	\caption{Expected value $\mathbb E_\alpha\big[e^{-\lambda T_1} F(\xi_1)\big]$ in dependency of $\alpha$ and $\lambda$ under the assumptions of Proposition \ref{prop:sufficientConditionsForGC} c) and assuming that $v>1$. In the shaded region we do not know whether the expected value is finite or infinite}
	\label{fig:phases}
\end{figure}

\begin{remark}
Let us consider the situation from Proposition \ref{prop:sufficientConditionsForGC} c). We choose $\beta$ and $g$ as in the proof of Proposition \ref{prop:sufficientConditionsForGC} c) and assume, as in the proof, that $v>1$. By Proposition \ref{at least <= 1} we have for all $\alpha>0$
\begin{equation*}
	\psi(\alpha) - \alpha = \min\big\{\lambda\in \R:\, \mathbb E_\alpha \big[e^{-\lambda T_1}F(\xi_1) \big] \leq 1 \big\}.
\end{equation*}
Using the known formula for the Laplace transform of an active period for this choice of $g$, see \cite{10.2307/1428137}, one obtains that the behavior of $\mathbb E_\alpha \big[e^{-\lambda T_1}F(\xi_1) \big]$ as a function of $\alpha$ and $\lambda$ is as depicted in Figure \ref{fig:phases}.
We call $\alpha$ ``good'' for $w$ if \eqref{Our condition} is satisfied. Notice that $\alpha$ is good for $w$ if and only if $1$ is good for $\alpha w$. 
A small calculation shows (where we denote the dependency on $w$ by another subscript)
\begin{align*}
&\partial_\alpha \widehat{\Gamma}_{1, T, \alpha w}(\text{the system is dormant at $0$}) \leq 0 \\
&\iff \mathbb E_{\widehat{\Gamma}_{1, T, \alpha w}}[N_\triangle]
\geq \mathbb E_{\widehat{\Gamma}_{1, T, \alpha w}}[N_\triangle|\text{the system is dormant at $0$}]
\end{align*}
where $N_\triangle : \mathbf N_f(\triangle) \to \N_0$, $\mu \mapsto \mu(\triangle)$ denotes the number of points in $\triangle$. Hence, it seems plausible that $ \widehat{\Gamma}_{1, T, \alpha w}(\text{the system is dormant at $0$})$ is decreasing in $\alpha$ for all $T>0$, which would imply (by Theorem \ref{Theorem: Equvialent conditions for goodness in the setting of path measures}) that the set of $\alpha$ that are good for $w$ is of the form $(0, b)$ or $(0, b]$ for some $b\in (0, \infty]$.
Then, if $b<\infty$, the function $\alpha \mapsto \lim_{T\to \infty} Z_{\alpha, T} e^{-\psi(\alpha)T}$ would be non-analytic (the limit exists irrespective whether \eqref{Our condition} is satisfied or not by Theorem \ref{Goodness criterions for alternating processes} and Proposition \ref{at least <= 1}).
\end{remark}

\section{Relations between $\psi(\alpha)$ and $\widehat{\Gamma}_{\alpha, \operatorname{st}}$}  
	\label{Fun calculations}
	In this short section we present a few formal calculations that relate the free energy $\psi$ 
	to certain expectations with respect to the tilted stationary measure $\widehat{\Gamma}_{\alpha, \operatorname{st}}$. 
	Although we exchange limits and integrals in an uncontrolled way in several places, we expect the 
	resulting formulae to be correct. They show that there is an intricate relationship between $\psi$ and 
	the expected value of several natural random variables with respect to $\widehat{\Gamma}_{\alpha, \operatorname{st}}$ which seems well 
	worth exploring further in the future. 
	
	We assume  that (A1)--(A3) are satisfied for all $\alpha$ and hence
	\begin{equation*}
	\mathbb E_\alpha\big[e^{-(\psi(\alpha) - \alpha)T_1} F(\xi_1) \big] = 1
	\end{equation*}
	for all $\alpha>0$. Multiplying the potential by $e^{c}$ for some $c\in \R$ yields with the number $N$ of customers in the first cluster
	\begin{equation*}
	\mathbb E_\alpha\big[e^{-(\psi(e^{c}\alpha) - \alpha)T_1} e^{c N} F(\xi_1) \big] = 1.
	\end{equation*}
	Assuming that $\psi$ is differentiable and that we may differentiate under the integral we obtain by differentiation with respect to $c$ 
	\begin{equation*}
	\alpha \psi'(\alpha) \mathbb E_\alpha\big[T_1 e^{-(\psi(\alpha) - \alpha)T_1} F(\xi_1)\big]= \mathbb E_\alpha\big[Ne^{-(\psi(\alpha) - \alpha)T_1} F(\xi_1)]
	\end{equation*}
	or short $\alpha \psi'(\alpha) \hat{\mathbb E}_{\alpha}[\widehat T_1] =    \hat{\mathbb E}_{\alpha}[\widehat N]$. In other words, $\psi(\alpha) = 1/\hat{\mathbb E}_{\alpha}[\hat d_1]$ determines the length of dormant periods in the reweighted process and $\alpha \psi'(\alpha)$ determines the number of points per unit of time in the reweighted process. Notice that 
	\begin{equation}
	\label{Equation: psi' is increase in point density}
	\psi'(\alpha) = \frac{\hat{\mathbb E}_\alpha[\widehat N]/\hat{\mathbb E}_\alpha[\widehat {T}_1]}{\mathbb E_\alpha[N]/\mathbb E_\alpha[T_1]}
	\end{equation}
	as $\alpha = \mathbb E_\alpha[N]/\mathbb E_\alpha[T_1]$ (this can be obtained from the previous considerations by choosing $w(t, x) = g(t)$ for $t\geq 0$ and $x\in \R^d$). 
	For a constant $c\geq 0$ we add $c\cdot g$ to $w$ (i.e. we add $c$ to $v$) and obtain
	\begin{equation*}
	\mathbb E_\alpha\bigg[e^{-(\psi(\alpha) + \alpha c - \alpha)T_1} \mathbb E_\mathcal W \Big[\prod_{i=1}^{N}v(t_i - s_i, X_{s_i, t_i}) + c \Big] \bigg]  = 1.
	\end{equation*}
	If we again assume that we may differentiate under the integral we obtain
	\begin{equation*}
	\alpha \mathbb E_\alpha\big[T_1e^{-(\psi(\alpha) - \alpha)T_1} F(\xi_1) \big] = \mathbb E_\alpha\big[Ne^{-(\psi(\alpha) - \alpha)T_1} \overline{F}(\xi_1)]
	\end{equation*}
	where
	\begin{equation*}
	\overline{F}(\xi_1) = \frac{1}{N}\sum_{j=1}^N \mathbb E_\mathcal W\Big[\prod_{i\neq j}v(t_i - s_i, X_{s_i, t_i}) \Big].
	\end{equation*}
	Combining both equalities yields
	\begin{equation}
	\label{eq: interesting formula}
	\psi'(\alpha) = \frac{ \mathbb E_\alpha\big[Ne^{-(\psi(\alpha) - \alpha)T_1} F(\xi_1)]}{ \mathbb E_\alpha\big[Ne^{-(\psi(\alpha) - \alpha)T_1} \overline{F}(\xi_1)]}.
	\end{equation}
	In other words, $\psi'(\alpha)$ is a measure for how much $F$ changes 
	on average if we randomly delete a point of $\xi_1$. 
	If there exists a measurable function $f:[0, \infty) \to (0, \infty)$ satisfying $\int_0^\infty (1+t)f(t) \, \mathrm dt< \infty$ such that $w(t, x) \leq f(t)$ for all $x\in \R^d$ and $t\geq 0$ (i.e. if the assumptions of Proposition \ref{prop:sufficientConditionsForGC} a) are satisfied) then we may choose $g = f/\|f\|_{L^1}$ and obtain $v \leq C$ with $C \coloneqq \|f\|_{L^1}$.
	This means that 
	$F(\xi_1) \leq C \bar F(\xi_1)$, and then \eqref{eq: interesting formula} implies a linear upper
	bound for the growth of $\psi(\alpha)$ with $\alpha$. It should be noted, however, that this can be obtained by an elementary calculation, using $e^{C \cdot c_{\alpha, T}} \sim e^{2\alpha C T - \alpha C \mathbb E[\tau_1]}$. Similarly, lower linear bounds can be derived. For the Fröhlich polaron, $\psi(\alpha) \sim g_0 \alpha^2$ and thus $\psi'(\alpha) \to \infty$ as $\alpha\to \infty$ (by convexity of $\psi$). Combining this with Equation \eqref{Equation: psi' is increase in point density}, we see that, as a consequence of the singularity of the potential, reweighting the point process leads to a relative increase of the number of customers per unit of time that diverges to $+\infty$ as $\alpha \to \infty$.

	By Theorem \ref{Theorem: Equvialent conditions for goodness in the setting of path measures} we have
	\begin{equation*}
	\hat{\mathbb E}_\alpha[\widehat T_1] = \mathbb E_\alpha\big[T_1e^{-(\psi(\alpha) - \alpha)T_1} F(\xi_1)\big] = \frac{1}{\psi(\alpha) \lim_{T\to \infty} Z_{\alpha, T} e^{-\psi(\alpha)T}}.
	\end{equation*}
	As seen in the proof of Proposition \ref{prop:sufficientConditionsForGC} b) and Corollary \ref{cor:polaron}, for the Fröhlich polaron this yields
	\begin{equation*}
	\hat{\mathbb E}_\alpha[\widehat T_1]= -\frac{1}{E_{\alpha}(0)|\langle \Omega, \Psi_\alpha\rangle|^2}
	\end{equation*}
	where $\Omega$ is the Fock-vacuum and $\Psi_\alpha$ and $E_\alpha(0)$ are the ground state and the ground state energy of the Hamiltonian of the Fröhlich polaron at total momentum zero. Notice that
	\begin{equation*}
		|\langle \Omega, \Psi_\alpha\rangle|^2 = \hat{\mathbb E}_\alpha[\hat d_1]/\hat{\mathbb E}_\alpha[\widehat T_1]
	\end{equation*}
	is the probability that we are dormant at a given point in time under the stationary measure $\widehat \Gamma_{\alpha, \operatorname{st}}$.
	We can also put the coupling parameter into the potential, to obtain 
	\begin{equation*}
	\mathbb E_1\Big[\alpha^Ne^{-(\psi(\alpha) - 1)T_1}F(\xi_1)\Big] = 1.
	\end{equation*}
	This identity might potentially be used in order to show analyticity of $\psi$ using the implicit function theorem.

\section*{Appendix: Polaron models and Polaron path measures} 
	
	Here we give an overview over polaron models and their connection to Polaron path measures. For 
	details and proofs we refer to  \cite{DySp20} and \cite{Mo06}, see also Chapters 5 and 6 of \cite{LHB11}.
		 
	The polaron describes a $d$-dimensional quantum particle coupled to a scalar Bosonic field, 
	e.g.\ the lattice vibrations of a polar crystal. Its Hamiltonian acts in the space 
	$L^2(\bbR^d) \otimes \caF$, where $\caF$ the Hilbert space for the Bose
	field, i.e.\ the symmetric 
	Fock space over $L^2(\bbR^d)$. The Hamiltonian is given by 
	\[
	H = \frac12 p^2 + \int_{\bbR^d} \omega(k) a^\ast(k) a(k) \, \dd k + \sqrt{\alpha} \int_{\bbR^d} 
	\frac{\hat \varrho(k)}{\sqrt{2 \omega(k)}}\big(\e{\ii k \cdot x} a(k) + \e{- \ii k \cdot x} a^\ast(k) \big) \, 
	\dd k.
	\] 
	Here $a^\ast(k)$, $a(k)$ are the creation and annihilation 
	operators of the free Bose field, respectively, satisfying the canonical commutation relations 
	$[a^\ast(k), a(k')] = \delta(k-k')$. The first term represents the momentum operator of the free particle, and 
	the second term is the energy of the free field, which is the  differential second quantization of the operator of 
	multiplication with $\omega$. 
 	The energy-momentum relation $\omega$ of the Bose field is assumed to be nonnegative, 
	strictly positive almost everywhere, continuous, and invariant under rotations. 
	The third term implements the coupling between particle and field, $x$ being the position operator 
	of the particle. 
	The function $\varrho$ is used to 'smear out' the coupling of the particle to the field. $\hat \varrho$ is 
	assumed to be rotation invariant and real-valued. The function 
	\[
	g(k) = \hat \varrho(k) / \sqrt{2 \omega(k)}
	\] 
	is 
	usually called the coupling function, and $\alpha$ is the coupling constant. Important special cases are the
	Fröhlich polaron where 
	\[
	\omega(k) = 1, \qquad g(k) = (\sqrt{2} \pi |k|)^{-1},
	\]
	and the Nelson model where 
	\[
	\omega(k) = \sqrt{k^2+m^2} \text{ with } m \geq 0, \quad \text{and } \quad g(k) = \bbone_{[\kappa,\Lambda]}(|k|) \omega(k)^{-1/2} \text{ with } 0 \leq \kappa < \Lambda < \infty.
	\] 
	$m$ is the mass of the Bosons, and $\kappa, \Lambda$ are the infrared and ultraviolet 
	cutoffs, respectively. They restrict the interaction of the particle with the field modes to those modes with 
	energy between $\kappa$ and $\Lambda$. 
		
	Under suitable assumptions on $\omega$ and $g$ (which are fulfilled for the two examples above), the 
	operator $H$ is self-adjoint and bounded below. Since the coupling of the 
	particle to the field is invariant under translations, $H$ commutes with the total momentum operator 
	\[
	P_{\rm tot} = p + P_{\rm f} = p + \int_{\bbR^d} k \, a^\ast(k) a(k) \, \dd k.
	\]
	Therefore, $H$ admits a fiber decomposition $H = U^\ast \int_{\bbR^d}^{\oplus} H(P) \, \dd P \,\, U$ 
	with a suitable unitary operator $U$, where for each $P \in \bbR^d$ the operator 
	\[
	H(P) = \frac12 (P - P_{\rm f})^2 +  \int_{\bbR^d} \omega(k) a^\ast(k) a(k) \, \dd k + \sqrt{\alpha} \int_{\bbR^d} 
	g(k) ( a(k) + a^\ast(k) ) \, \dd k.
	\] 
	now acts only on Fock space. $H(P)$ is bounded below and self-adjoint whenever $H$ is, and the map that takes 
	$P \in \bbR^d$ to the 
	bottom $E(P)$ of the spectrum of $H(P)$ is the energy-momentum relation for the particle. By the rotation invariance
	of $\omega$ and $g$, $E(P) = E_{\rm r}(|P|)$ only depends on $|P|$, and $E''_{\rm r}(0)$ is the inverse of the 
	effective mass of the particle interacting with the Bose field. 
	
	Polaron path measures are related to polaron models by a Feynman-Kac formula: 
	the particle Hamiltonian $\frac12 p^2$ is the generator of Brownian motion
	$\caW$, and the field Hamiltonian 
	$\int \omega(k) a^\ast(k) a(k) \, \dd k$ is unitarily equivalent to the generator of an infinite dimensional 
	Ornstein-Uhlenbeck process $\caG$. Its probability distribution is supported on distribution-valued functions 
	$s \mapsto \phi_s$, and its covariance function is given by 
	\[
	 \bbE_\caG(\phi_s(u) \phi_t(v))  = \int \hat{u}(k) \frac{1}{2 \omega(k)} \e{-|t-s| \omega(k)} \overline{ \hat v(k)} 
	\, \dd k
	\]
	for suitable test functions $u,v$.   
	Much like in the ordinary Feynman-Kac formula, this allows to write matrix elements 
	$\langle \Psi, \e{-tH} \Phi \rangle$ as integrals with respect to the measure $\caG \otimes \caW$. When 
	$\Omega$ is the Fock vacuum and $f_1,f_2 \in L^2(\bbR^d)$, this leads to the equality 
	\[
	\langle  f_1 \otimes \Omega, \e{-tH} f_2 \otimes \Omega \rangle 
	= 
	\int \dd x f_1(x) \int \mathcal W^x(\dd \mathbf x) \int \caG(\dd \phi)
	\exp \Big(- \sqrt{\alpha} \int_0^t \phi_s(\rho(\cdot - \mathbf x_s)) \, \dd s \Big) f_2(\mathbf x_t),
	\]
	where $\caW^x$ is Brownian motion started at $x$. 
	Since the exponent is linear in the field variable $\phi$, the Gaussian integral can be carried out explicitly, 
	with the result 
	\begin{equation} \label{FKF} 
	\langle  f_1 \otimes \Omega, \e{-tH} f_2 \otimes \Omega \rangle = 
	\int \dd x f_1(x) \int  W^x(\dd \mathbf x) \e{\frac{\alpha}{2} \int_0^t \dd r \int_0^t \dd s \,
	 w(r-s,\mathbf x_r - \mathbf x_s) } f_2(\mathbf x_t),
	\end{equation} 
	where 
	\[ 
	w(r-s,\mathbf x_r - \mathbf x_s) = 
	\bbE_{\caG} \big(\phi_r\big(\varrho(\cdot - \mathbf x_r)\big) \phi_s\big(\varrho(\cdot - \mathbf x_s)\big)\big) 
	=  \int |g(k)|^2 \e{\ii k \cdot (\mathbf x_r - \mathbf x_s)} \e{-\omega(k) |r-s|}.
	\]	
	This establishes the connection between the Polaron models and Polaron path measures. 
	The formal choice $f_1 = \delta(0)$ and $f_2 = 1$ in \eqref{FKF} 
	corresponds to the matrix element  $\langle \Omega, e^{-tH(0)}\Omega \rangle$, 
	see e.g.\ \cite{DySp20}, where also expressions for $\langle \Omega, e^{-tH(P)}\Omega \rangle$ are derived for arbitrary 
	$P$. 
	
	For the Fröhlich polaron in three dimensions, an explicit computation leads to 
	$w(t,x) = \frac{1}{|x|} \e{-|t|}$. For the massless Nelson model (i.e. $m=0$) in three dimensions, one obtains 
	\[
	w(t,x) = 4 \pi \int_\kappa^\Lambda \e{-r|t|} \frac{\sin(r|x|)}{|x|} \dd r,
	\]
	When $\kappa = 0$, i.e.\ when the particle is allowed to interact with low energy field modes, this $w$ decays
	like $|t|^{-2}$ for large $t$. \\[2mm]

	{\bf Acknowledgement:} We thank Herbert Spohn for useful comments on an earlier version of this paper.

\newcommand{\etalchar}[1]{$^{#1}$}


\end{document}